\documentclass[11pt,reqno]{amsart}
\usepackage{mathrsfs}
\usepackage{url}
\usepackage{mathtools}
\usepackage{latexsym,epsfig,amssymb,amsmath,amsthm,color,url,bm}
\usepackage[inline,shortlabels]{enumitem}
\usepackage{hyperref}
\usepackage[foot]{amsaddr}
\usepackage{amsmath,amsbsy}
\usepackage{mwe}
\RequirePackage[numbers]{natbib}
\usepackage{mathptmx}
\usepackage[text={16cm,24cm}]{geometry}

\allowdisplaybreaks 
 \numberwithin{equation}{section}
\newtheorem{theorem}{Theorem}[section]

\newtheorem{lemma}[theorem]{Lemma}

\newcommand\Item[1][]{%
  \ifx\relax#1\relax  \item \else \item[#1] \fi
  \abovedisplayskip=0pt\abovedisplayshortskip=0pt~\vspace*{-\baselineskip}}

\theoremstyle{definition}

\theoremstyle{definition}
\newtheorem{remark}[theorem]{Remark}

\DeclareMathOperator{\Prob}{\mathbf{P}}

\title[Learning models on trees witih majority update policy]{Learning models on rooted regular trees with majority update policy: convergence and phase transition}
\date{}
\author{Moumanti Podder, Anish Sarkar}
\address{Moumanti Podder, Indian Institute of Science Education and Research (IISER) Pune, Dr.\ Homi Bhabha Road, Pashan, Pune 411008, Maharashtra, India.}
\address{Anish Sarkar, Indian Statistical Institute, 7 S.\ J.\ S.\ Sansanwal Marg, New Delhi 110016, India.}
\email{moumanti@iiserpune.ac.in}
\email{anish.sarkar@gmail.com}

\begin{document}
\bibliographystyle{plainnat}

\begin{abstract}
We study a model of social learning on rooted regular trees. An agent is stationed at each vertex of $\mathbb{T}_{m}$, the rooted tree in which each vertex has precisely $m$ children, and at any time-step $t \in \mathbb{N}_{0}$, the agent is allowed to select one of two available technologies: $B$ and $R$. Let the technology chosen by the agent at vertex $v$ of $\mathbb{T}_{m}$, at time-step $t$, be $C_{t}(v)$. We begin with the i.i.d.\ collection $\{C_{0}(v): v \in \mathbb{T}_{m}\}$, where $C_{0}(v)=B$ with probability $\pi_{0}$. During the epoch $t$, the agent at vertex $v$ performs an experiment that results in success with probability $p_{B}$ if $C_{t}(v)=B$, and with probability $p_{R}$ if $C_{t}(v)=R$. If the children of $v$ are denoted $v_{1}, \ldots, v_{m}$, the agent at $v$ updates their technology to $C_{t+1}(v)=B$ if the number of successes among all $v_{i}$ (where $i \in \{1,2,\ldots,m\}$) with $C_{t}(v_{i})=B$ exceeds, strictly, the number of successes among all $v_{j}$ (where $j \in \{1,2,\ldots,m\}$) with $C_{t}(v_{j})=R$. If these two numbers are equal, then the agent at $v$ sets $C_{t+1}(v)=B$ with probability $1/2$. In all other cases, $C_{t+1}(v)=R$. We show that $\{C_{t}(v): v \in \mathbb{T}_{m}\}$ is i.i.d.\ as well, with $C_{t}(v)=B$ with probability $\pi_{t}$, where the sequence $\{\pi_{t}\}_{t \in \mathbb{N}_{0}}$ converges to a fixed point $\pi$, in $[0,1]$, of a function $g_{m}$. We show that, for $m \geqslant 3$, there exists a $p(m) \in (0,1)$ such that $g_{m}$ has the unique fixed point, $1/2$, when $p \leqslant p(m)$, and three distinct fixed points, of the form $\alpha$, $1/2$ and $1-\alpha$, for some $\alpha \in [0,1/2)$, when $p > p(m)$. When $m=3$, $p_{B}=1$ and $p_{R} \in [0,1)$, we show that the function $g_{3}$
\begin{enumerate*}
\item has a unique fixed point, $1$, when $p_{R} < \sqrt{3}-1$, 
\item has two distinct fixed points, one of which is $1$, when $p_{R} = \sqrt{3}-1$,
\item and has three distinct fixed points, one of which is $1$, when $p_{R} > \sqrt{3}-1$.
\end{enumerate*}
When $g_{m}$ has multiple fixed points, we also specify which of these fixed points $\pi$ equals, depending on $\pi_{0}$. Finally, for $m=2$, we describe the behaviour of $g_{3}$ for \emph{all} values of $p_{B}$ and $p_{R}$.
\end{abstract}



\maketitle

\section{Introduction}\label{sec:intro}
\subsection{An overview of our learning model}\label{subsec:overview}
For any positive integer $m$ with $m \geqslant 2$, let $\mathbb{T}_{m}$ indicate the infinite rooted tree, with root $\phi$, in which each vertex has precisely $m$ children. Let $V(\mathbb{T}_{m})$ denote the set of all vertices of $\mathbb{T}_{m}$, and imagine an agent occupying each vertex of $V(\mathbb{T}_{m})$. We consider a \emph{learning model} in which each agent, at each time-step (where time is indexed by the set $\mathbb{N}_{0}$ of non-negative integers), has to adopt one of two possible technologies that we henceforth refer to as \emph{colours} or \emph{states} and denote by the letters $B$ and $R$. The state of the agent at a vertex $v \in V(\mathbb{T}_{m})$ at time-step $t$, for each $t \in \mathbb{N}_{0}$, is denoted by $C_{t}(v)$. Conditioned on the states of the agents at all the vertices of $V(\mathbb{T}_{m})$ at time-step $t$, the (random) state $C_{t+1}(v)$ to which the agent at $v$ updates itself at time-step $t+1$ has a probability distribution that is a function of $C_{t}(v_{1}), C_{t}(v_{2}), \ldots, C_{t}(v_{m})$, where $v_{1}, v_{2}, \ldots, v_{m}$ are the children of $v$. 

Such a model can be viewed from a number of different perspectives, such as
\begin{enumerate}
\item as a discrete-time \emph{interacting system of particles} on $\mathbb{T}_{m}$ with infinitely many changes allowed to happen at each time-step,
\item as a model for \emph{social learning}, for understanding the \emph{diffusion of technologies} throughout an infinite population of agents,
\item as a \emph{probabilistic finite state tree automaton} on $\mathbb{T}_{m}$.
\end{enumerate}
We delve deeper into each of the above motivating reasons for the study of such models in \S\ref{subsec:motivations_literature}.

Each model studied in this paper admits a crucial component known as a \emph{policy function}: it governs how the agent at $v$, for each $v \in V(\mathbb{T}_{m})$, \emph{learns} from the states $C_{t}(v_{i})$, $i \in \{1,2,\ldots,m\}$, of the agents situated at the children $v_{1}, v_{2}, \ldots, v_{m}$ of $v$ at time-step $t$, and subsequently updates its own state to $C_{t+1}(v)$ at time-step $t+1$. When such a model is viewed as a \emph{probabilistic tree automaton} (henceforth abbreviated as a \emph{PTA}), as explained in \S\ref{subsec:motivations_literature}, the policy function is referred to as the \emph{stochastic update rule} associated with that PTA. We are concerned with studying the \emph{absolute majority policy} function, which we describe in detail in \S\ref{subsec:abs_maj}.

\subsection{The absolute majority policy}\label{subsec:abs_maj}
\sloppy In this model, the agent occupying vertex $v$, for each $v \in V(\mathbb{T}_{m})$, is endowed with two random variables, $X_{t}(v)$ and $Y_{t}(v)$, for each time-step $t \in \mathbb{N}_{0}$. For each $v \in V(\mathbb{T}_{m})$ and each $t \in \mathbb{N}_{0}$, the random variable $X_{t}(v)$ captures the \emph{outcome} of the (random) experiment performed by the agent located at $v$, during the time-step or \emph{epoch} $t$, using the technology $C_{t}(v)$ that this agent is currently equipped with. This experiment has two possible outcomes: \emph{success} and \emph{failure}, so that $X_{t}(v)$ is a Bernoulli random variable that acts as an indicator for the event of success. We let $X_{t}(v)$ follow Bernoulli$(p_{B})$ conditioned on $C_{t}(v)=B$, and we let $X_{t}(v)$ follow Bernoulli$(p_{R})$ conditioned on $C_{t}(v)=R$, where $p_{B}, p_{R} \in [0,1]$ are pre-assigned parameters in our model. The random variable $Y_{t}(v)$ follows Bernoulli$\left(\frac{1}{2}\right)$ for each $v \in V(\mathbb{T}_{m})$ and each $t \in \mathbb{N}_{0}$. It is assumed that the entire collection $\left\{X_{t}(v): v \in V(\mathbb{T}_{m})\right\}\bigcup\left\{Y_{t}(v): v \in V(\mathbb{T}_{m})\right\}$ of random variables is independent for \emph{each} time-step $t \in \mathbb{N}_{0}$, and the collection $\left\{Y_{t}(v): v \in V(\mathbb{T}_{m})\right\}$ is independent of $\left\{C_{s}(u): u \in V(\mathbb{T}_{m}), s \in \{0,1,\ldots,t\}\right\}$ for each $t \in \mathbb{N}_{0}$. 

Conditioned on $\left\{C_{t}(u): u \in V(\mathbb{T}_{m})\right\}$, the agent at $v$ updates its state to $C_{t+1}(v)$ according to the following rule, where $v_{1}, v_{2}, \ldots, v_{m}$ denote the children of $v$:
\begin{equation}\label{abs_maj_rule}
 C_{t+1}(v) = 
  \begin{cases} 
   B & \text{if } \sum_{i=1}^{m}X_{t}(v_{i})\mathbf{1}_{C_{t}(v_{i})=B} > \sum_{i=1}^{m}X_{t}(v_{i})\mathbf{1}_{C_{t}(v_{i})=R}, \\
   B & \text{if } \sum_{i=1}^{m}X_{t}(v_{i})\mathbf{1}_{C_{t}(v_{i})=B} = \sum_{i=1}^{m}X_{t}(v_{i})\mathbf{1}_{C_{t}(v_{i})=R} \text{ and } Y_{t}(v)=1, \\
   R & \text{otherwise.}
  \end{cases}
\end{equation}
In words, this can be described as follows: the agent at the vertex $v$ counts the number $\sum_{i=1}^{m}X_{t}(v_{i})\mathbf{1}_{C_{t}(v_{i})=B}$ of successful experiments performed, during epoch $t$, by all those agents that occupy the children of $v$ \emph{and} that are in state $B$ at time-step $t$, and likewise, it also counts the number $\sum_{i=1}^{m}X_{t}(v_{i})\mathbf{1}_{C_{t}(v_{i})=R}$ of successful experiments performed, during epoch $t$, by all those agents that occupy the children of $v$ \emph{and} that are in state $R$ at time-step $t$. If the former count strictly exceeds the latter, then the agent at $v$ updates its state to $B$ at time-step $t+1$, while if the latter strictly exceeds the former, then the agent at $v$ updates its state to $R$ at time-step $t+1$. If the two counts are equal, then the agent at $v$ tosses a fair coin in order to break the tie. If the coin lands a head, indicated by $Y_{t}(v)=1$, then the agent at $v$ updates its state to $B$ at time-step $t+1$, and if the coin lands a tail, it updates its state to $R$ at time-step $t+1$.

We begin the process at time-step $t=0$ with an i.i.d.\ assignment of states to the vertices of $\mathbb{T}_{m}$, i.e.\ the collection $\left\{C_{0}(v): v \in V(\mathbb{T}_{m})\right\}$ is i.i.d.\ with
\begin{equation}
C_{0}(v)=B \text{ with probability } \pi_{0} \text{ and } C_{0}(v)=R \text{ with probability } 1-\pi_{0}\label{initial}
\end{equation}
for each $v \in V(\mathbb{T}_{m})$. Letting $\nu_{t}$ denote the joint law of $\left\{C_{t}(v): v \in V(\mathbb{T}_{m})\right\}$ for each $t \in \mathbb{N}_{0}$, this paper is concerned with answering the question: does $\left\{\nu_{t}\right\}_{t \in \mathbb{N}_{0}}$ converge as $t \rightarrow \infty$, and if yes, what is its distributional limit? We investigate this question for \emph{all} values of $(p_{B},p_{R}) \in [0,1]^{2}$ when $m=2$, whereas when $m \geqslant 3$, we consider two \emph{regimes} of values of the parameter-pair $(p_{B},p_{R})$. The first of these regimes is where $p_{B}=p_{R}=p$ for some $p \in [0,1]$, i.e.\ an agent, located at any vertex of $\mathbb{T}_{m}$, stands the same chance of achieving a successful outcome in the experiment performed during epoch $t$ irrespective of whether its state at time-step $t$ is $B$ or $R$. We present a full description of what happens in such a scenario, as the common value $p$ of the two parameters is allowed to vary over $[0,1]$. The complement of this regime is where $p_{B} \neq p_{R}$, but when $m \geqslant 3$, the analysis for $p_{B} \neq p_{R}$, without additional assumptions, quickly becomes intractable, and it becomes difficult to conclude anything about the distributional limit(s) of $\left\{\nu_{t}\right\}_{t \in \mathbb{N}_{0}}$ (depending on the initial distribution $\nu_{0}$ given by \eqref{initial}, and on the values of $p_{B}$ and $p_{R}$). We, therefore, focus on the special case of $p_{B}=1$ and $p_{R} \in [0,1)$, and we demonstrate, via the relatively less unwieldy examples of $m=2$ and $m=3$, that the conclusion is rather non-intuitive and further investigations are warranted. This special case can be considered ``extremal" in the sense that, a vertex that is in state $B$ is \emph{sure} to achieve success when it performs its experiment (using technology $B$), since $p_{B}=1$, whereas the experiment performed (using technology $R$) by a vertex in state $R$ has a positive probability of resulting in a failure. Intuition seems to suggest that in such a scenario, eventually, every vertex of $\mathbb{T}_{m}$ would opt for the (evidently superior) technology $B$, but as it turns out (see \S\ref{sec:p_{B}=1}), this is not necessarily the case, and the distributional limit depends heavily on the value of $p_{R}$.

\subsection{Main results}\label{subsec:main_results}
We begin with a result that is applicable to \emph{all} values of the parameter-pair $(p_{B},p_{R})$:
\begin{theorem}\label{thm:main_general}
Recall that $\nu_{t}$ indicates the joint law of $\left\{C_{t}(v): v \in V(\mathbb{T}_{m})\right\}$ for each $t \in \mathbb{N}_{0}$, and recall that we begin from an i.i.d.\ assignment of states (from $\{B,R\}$) to the vertices of $\mathbb{T}_{m}$, as described in \eqref{initial}. The following are true:
\begin{enumerate}
\item \label{gen:1} For each $t \in \mathbb{N}$, the random variables $C_{t}(v)$, for $v \in V(\mathbb{T}_{m})$, are i.i.d.\ with
\begin{equation}
C_{t}(v)=B \text{ with probability } \pi_{t} \text{ and } C_{t}(v)=R \text{ with probability } 1-\pi_{t},\label{time_t}
\end{equation}
for each $v \in V(\mathbb{T}_{m})$.
\item \label{gen:2} The sequence $\{\nu_{t}\}_{t \in \mathbb{N}_{0}}$ converges if and only if the sequence $\{\pi_{t}\}_{t \in \mathbb{N}_{0}}$ converges, as $t \rightarrow \infty$, and in that case, if $\pi = \lim_{t \rightarrow \infty}\pi_{t}$, then $\{\nu_{t}\}_{t \in \mathbb{N}_{0}}$ converges to $\nu$, where $\nu$ is the joint law of the collection $\left\{C_{\infty}(v): v \in V(\mathbb{T}_{m})\right\}$ of i.i.d.\ random variables $C_{\infty}(v)$, where
\begin{equation}
C_{\infty}(v) = B \text{ with probability } \pi \text{ and } C_{\infty}(v) = R \text{ with probability } 1-\pi\label{limit_pi}
\end{equation}
for each $v \in V(\mathbb{T}_{m})$.
\item \label{gen:3} The limit $\pi = \lim_{t \rightarrow \infty}\pi_{t}$, when it exists, is a fixed point of the function $g_{m}: [0,1] \rightarrow [0,1]$ defined as
\begin{equation}
g_{m}(x) = \sum_{k=0}^{m}f_{m}(k) {m \choose k} x^{k}(1-x)^{m-k},\label{g_{m}^{abs}}
\end{equation}
in which
\begin{equation}\label{f_{m}^{abs}}
f_{m}(k) = \Prob\left[A_{k} > B_{m-k}\right] + \frac{1}{2}\Prob\left[A_{k} = B_{m-k}\right],
\end{equation} 
where the random variable $A_{k}$ follows Binomial$(k,p_{B})$, the random variable $B_{m-k}$ follows Binomial$(m-k,p_{R})$, and $A_{k}$ and $B_{m-k}$ are independent of each other.
\end{enumerate}
\end{theorem}
From the last assertion made in the statement of Theorem~\ref{thm:main_general}, it becomes apparent that we need to answer the following questions:
\begin{enumerate}
\item Does the function $g_{m}$ have a unique fixed point in $[0,1]$?
\item If $g_{m}$ has multiple fixed points in $[0,1]$, which of these are attractive and which are repulsive?
\end{enumerate}

Let us now make the following simple observation: 
\begin{remark}\label{rem:unique_fixed_point_attractive}
When $g_{m}$ has a unique fixed point, say $\pi$, in $[0,1]$, it must be attractive, i.e.\ no matter what value $\pi_{0} \in [0,1]$ we choose in \eqref{initial}, the sequence of laws $\{\nu_{t}\}_{t \in \mathbb{N}_{0}}$ converges, as $t \rightarrow \infty$, to $\nu$ defined by \eqref{limit_pi}. 
\end{remark}
\begin{proof}
Let us consider any subsequence $\left\{\pi_{t_{k}}\right\}_{k \in \mathbb{N}}$ of the sequence $\{\pi_{t}\}_{t \in \mathbb{N}_{0}}$ (defined in \eqref{time_t}). Since this subsequence is bounded, it has a further subsequence, say $\left\{\pi_{t_{k_{\ell}}}\right\}_{\ell \in \mathbb{N}}$, that converges as $\ell \rightarrow \infty$. From \eqref{gen:3} of Theorem~\ref{thm:main_general}, we know that $\lim_{\ell \rightarrow \infty}\pi_{t_{k_{\ell}}}$ has to be a fixed point of $g_{m}$, and as $g_{m}$ has a unique fixed point, $\pi$, in $[0,1]$, we must have $\lim_{\ell \rightarrow \infty}\pi_{t_{k_{\ell}}}=\pi$. This allows us to conclude that the entire sequence $\{\pi_{t}\}_{t \in \mathbb{N}_{0}}$, in fact, converges to $\pi$ as $t\rightarrow \infty$, and the rest of the conclusion in Remark~\ref{rem:unique_fixed_point_attractive} follows from \eqref{gen:2} of Theorem~\ref{thm:main_general}.
\end{proof} 

Before moving on to higher values of $m$, we state the following result concerning $m=2$ and \emph{all} possible values of $p_{B}$ and $p_{R}$:
\begin{theorem}\label{thm:main_m=2}
The function $g_{2}$ is strictly convex for $p_{R} > p_{B}$, strictly concave for $p_{R} < p_{B}$, and linear for $p_{B}=p_{R}$. In each of these cases, except when $p_{B}=p_{R}=1$, the function $g_{2}$ has a unique fixed point in $(0,1)$. When $p_{B}=p_{R}=1$, the entire interval $[0,1]$ constitutes the set of fixed points of $g_{2}$.
\end{theorem}
From Remark~\ref{rem:unique_fixed_point_attractive}, it is evident that for $m=2$ and when $(p_{B},p_{R}) \neq (1,1)$, the sequence $\{\nu_{t}\}_{t \in \mathbb{N}_{0}}$ converges to $\nu$, as $t \rightarrow \infty$, where $\nu$ is as defined by \eqref{limit_pi}, with $\pi$ being the unique fixed point of $g_{2}$. For $m=2$ and $p_{B}=p_{R}=1$, for each $\pi_{0} \in [0,1]$ (where $\pi_{0}$ is as defined in \eqref{initial}), the law $\nu_{t}$ is the same as $\nu_{0}$ for every $t \in \mathbb{N}$, so that the limit equals $\nu_{0}$ itself. 

We now come to the two specific regimes of values of the parameter-pair $(p_{B},p_{R})$ that, as mentioned towards the end of \S\ref{subsec:abs_maj}, are considered in this paper. When $p_{B}=p_{R}=p$ for some $p \in [0,1]$, the learning model with the absolute majority policy described in \S\ref{subsec:abs_maj} admits the following phase transition phenomenon:
\begin{theorem}\label{thm:main_equal}
Recall the function $g_{m}$ defined in \eqref{g_{m}^{abs}}, and consider $p_{B}=p_{R}=p$ in the definition of the function $f_{m}$ in \eqref{f_{m}^{abs}}. For each $m \in \mathbb{N}$ with $m \geqslant 2$, there exists $p(m) \in (0,1)$ such that 
\begin{enumerate}
\item \label{thm:equal_part_1} for all $p \leqslant p(m)$, the function $g_{m}$ has a unique fixed point in $[0,1]$, which is $1/2$; 
\item \label{thm:equal_part_2} for all $p > p(m)$, the function $g_{m}$ has precisely three fixed points in $[0,1]$, of the form $\alpha$, $1/2$ and $1-\alpha$, where $\alpha=\alpha(p,m) < 1/2$ is some function of $p$ and $m$).
\end{enumerate}
Recall $\pi_{t}$ from \eqref{time_t}, for each $t \in \mathbb{N}$. For \emph{every} choice of the initial distribution (i.e.\ for every choice of $\pi_{0}$, as defined in \eqref{initial}), the sequence $\{\pi_{t}\}_{t \in \mathbb{N}_{0}}$ converges, as $t \rightarrow \infty$, to $\pi$ (as defined in the statement of Theorem~\ref{thm:main_general}). Here, $\pi=1/2$ for each $\pi_{0} \in [0,1]$ when $p \leqslant p(m)$, whereas when $p > p(m)$, we have 
\begin{enumerate*}
\item $\pi=\alpha$ when $\pi_{0} \in [0,1/2)$, 
\item $\pi=1/2$ when $\pi_{0}=1/2$,
\item and $\pi=1-\alpha$ when $\pi_{0} \in (1/2,1]$.
\end{enumerate*}
\end{theorem}
\begin{remark}\label{rem:p_{B}=p_{R}_unique_regime_already_proved}
Note that in the regime where $p \leqslant p(m)$, since $g_{m}$ has a unique fixed point in $[0,1]$, and this equals $1/2$, hence, by Remark~\ref{rem:unique_fixed_point_attractive}, we already know that $\pi=1/2$ for each $\pi_{0} \in [0,1]$. Therefore, in the last assertion made in the statement of Theorem~\ref{thm:main_equal}, we need only focus on the regime where $p > p_{m}$.
\end{remark}

In particular, we show, via analytical methods in \S\ref{subsec:small_m}, that for $m=3$, the threshold $p(m) \approx 0.557507$, whereas for $m=4$, the threshold $p(m) \approx 0.42842$.

When $p_{B}=1$ and $p_{R} \in [0,1)$, our learning model is devoid of any phase transition for $m=2$ (which follows as a special case of Theorem~\ref{thm:main_m=2}), whereas a striking and rather non-intuitive phenomenon of phase transition is observed when $m=3$, as stated below:
\begin{theorem}\label{thm:unequal}
When $m=3$, $p_{B}=1$ and $p_{R} \in [0,1]$, the function $g_{m}$ defined in \eqref{g_{m}^{abs}} 
\begin{enumerate}
\item has a unique fixed point in $[0,1]$, which is $1$, when $p_{R} \in [0,\sqrt{3}-1)$,
\item has two distinct fixed points in $[0,1]$, namely $\alpha=2/3-1/\sqrt{3}$ and $1$, when $p_{R} = \sqrt{3}-1$,
\item and has three distinct fixed points in $[0,1]$, one of which is $1$, when $p_{R} \in (\sqrt{3}-1,1]$.
\end{enumerate}
Let $\pi_{0}$ be as defined in \eqref{initial}, and recall $\pi=\lim_{n \rightarrow \infty}\pi_{n}$, where $\pi_{n}$ is as defined in \eqref{time_t}. When $p_{R} \in [0,\sqrt{3}-1)$, we have $\pi=1$. When $p_{R} = \sqrt{3}-1$, we have $\pi=\alpha$ when $\pi_{0} \in [0,\alpha)$ and $\pi=1$ when $\pi_{0} \in (\alpha,1)$, and $\pi=\pi_{0}$ when $\pi_{0} \in \{\alpha,1\}$. When $p_{R} \in (\sqrt{3}-1,1)$, let the two fixed points \emph{other than} $1$ be denoted by $\alpha_{1}=\alpha_{1}(p_{R})$ and $\alpha_{2}=\alpha_{2}(p_{R})$, with $0 < \alpha_{1} < \alpha_{2} < 1$. In this case, we have $\pi=\alpha_{1}$ whenever $\pi_{0} \in [0,\alpha_{2})$, and $\pi=1$ whenever $\pi_{0} \in (\alpha_{2},1]$, whereas $\pi=1$ when $\pi_{0}=\alpha_{2}$.
\end{theorem}
The reason we describe this result as non-intuitive is as follows. The value $p_{B}=1$ ensures that \emph{every} experiment performed using the technology labeled $B$ is bound to result in a success. Evidently, the technology labeled $B$ is superior to the technology labeled $R$, unless $p_{R}=1$. This would seem to suggest that \emph{eventually}, the state of each vertex of $\mathbb{T}_{m}$ would be updated to $B$ and would remain that way, but this is not the case when $p_{R}$ is sufficiently large (i.e.\ for $p_{R} \geqslant \sqrt{3}-1$). The phase transition exhibited by our model in Theorem~\ref{thm:unequal}, therefore, seems unusual.

\subsection{Motivation and a brief discussion of relevant literature}\label{subsec:motivations_literature} 
As alluded to in \S\ref{subsec:overview}, there is a myriad of perspectives from which the study of our model can be motivated. For any two vertices $u, v \in V(\mathbb{T}_{m})$, with $u$ the parent of $v$, let us refer to the agent located at $v$ as an \emph{influencing neighbour} to the agent located at $u$. Our model can then be viewed as an interacting particle system in which each agent, at each time-step, is allowed to update their policy or protocol on some matter by taking note of the policies or protocols adopted, and the outcomes (such as rewards or penalties) obtained as a consequence, by its influencing neighbours at the \emph{previous} time-step. Each agent can adopt one of two available policies (such as a \emph{liberal} stance and a \emph{conservative} stance) at the beginning of time-step $t$, following which it executes a plan of action (as decreed by the adopted policy) \emph{during} the epoch $t$. Such a plan of action (which is a random experiment) has two possible outcomes: a \emph{desirable} one and an \emph{undesirable} one. At the end of time-step $t$, each agent counts the number of its influencing neighbours who adopted the liberal stance at the beginning of time-step $t$ and who were then rewarded with the desirable outcome, and the number of its influencing neighbours who adopted the conservative stance at the beginning of time-step $t$ and who were then rewarded with the desirable outcome. If the former exceeds the latter, the agent adopts the liberal stance at the beginning of time-step $t+1$, whereas if the latter exceeds the former, the agent adopts the conservative stance at the beginning of time-step $t+1$. If the two counts are equal, the agent decides which stance to adopt by tossing a fair coin and thereby breaking the tie.

The interpretation presented above is reminiscent of the well-known and well-studied \emph{voter model}. In the classical \emph{linear voter model} (see, for instance, \cite{holley1975ergodic} that introduced the voter model as a \emph{continuous time proximity process}, as well as \cite{liggett1985interacting}, \cite{liggett2013stochastic} and \cite{liggett1997stochastic} for a detailed review of old and new developments in this topic of research), individuals or agents occupy the elements (referred to as \emph{sites}) of an arbitrary countable set $S$, and each agent, at any given point in time, can have one of two possible opinions on an issue. For each site $x \in S$, at exponential times with a specified rate, the agent located at $x$ chooses a site $y \in S \setminus \{x\}$ with probability $p(x,y)$ (where $p(x,y) \geqslant 0$ for each $y \in S\setminus\{x\}$ and $\sum_{y \in S \setminus \{x\}}p(x,y)=1$) and adopts the opinion being currently held by the agent at $y$.

Yet another interpretation of the voter model, as suggested in \cite{clifford1973model}, is that of two species, fairly matched, competing for territory, and the invasion, by one species, of the territory belonging to the other. In our model, each vertex $v$ of $\mathbb{T}_{m}$ can be viewed as a city, and $B$ and $R$ can be thought of as two belligerent armies contesting each other. At any time-step $t$, the event $C_{t}(v)=B$ indicates that army $B$ retains dominion over the city $v$ at the \emph{beginning} of time-step $t$, but that there are ongoing clashes between $B$ and $R$ \emph{during} the epoch $t$. Conditioned on $C_{t}(v)=B$, the event $X_{t}(v)=1$ indicates that the \emph{defenders}, i.e.\ the armed forces of $B$, are able to defeat the \emph{invaders}, i.e.\ the armed forces of $R$, in the battle that takes place in $v$ during the epoch $t$. It is not only reasonable, but even practical, to assume that, in the event of $X_{t}(v)=0$, even though the invaders clinch a victory over the defenders, they incur heavy losses, and those among them that remain alive are not able to travel from $v$ to anywhere else for further attempts at conquests. A similar interpretation is true if $C_{t}(v)=R$. Letting $v_{1}, v_{2}, \ldots, v_{m}$ denote the children of $v$, one may imagine them to be, in some sense, ``border cities" or ``outposts" to the ``main city" $v$, with routes leading \emph{to} $v$ (along which armies can travel). From each outpost $v_{i}$ where the defenders $B$ defeat the invaders $R$ in the battle that takes place during the epoch $t$, a battalion of soldiers loyal to $B$ is dispatched to $v$ by the end of the epoch $t$. Likewise, from each $v_{i}$ where the defenders $R$ defeat the invaders $B$ in the battle that takes place during the epoch $t$, a battalion of soldiers loyal to $R$ is dispatched to $v$ by the end of the epoch $t$. The inequality $\sum_{i=1}^{m}X_{t}(v_{i})\mathbf{1}_{C_{t}(v_{i})=B} > \sum_{i=1}^{m}X_{t}(v_{i})\mathbf{1}_{C_{t}(v_{i})=R}$ then indicates that the soldiers who are loyal to $B$ and who travel to $v$ from its outposts outnumber the soldiers who are loyal to $R$ and who travel to $v$ from its outposts, resulting in $v$ becoming occupied by $B$ by the beginning of the epoch $t+1$. On the other hand, when $\sum_{i=1}^{m}X_{t}(v_{i})\mathbf{1}_{C_{t}(v_{i})=B} = \sum_{i=1}^{m}X_{t}(v_{i})\mathbf{1}_{C_{t}(v_{i})=R}$, the number of soldiers, loyal to $B$, who travel to $v$ from its outposts equals the number of soldiers, loyal to $R$, who travel to $v$ from its outposts, and in this case, $v$ is equally likely to be conquered by either of the two armies by the beginning of epoch $t+1$.  

Interacting particle systems are indispensable for modeling complex phenomena, in natural and social sciences, that involve a very large number of interrelated components (such as the flow of traffic on highways, interaction among the constituents of a cell, opinion dynamics, spread of epidemics or fires, reaction diffusion systems, crystal surface growth, chemotaxis etc.). For instance, the \emph{biased voter model}, a generalization of the voter model described above, has been studied in \cite{williams1972stochastic}, \cite{bramson1980williams}, \cite{bramson1981williams} and \cite{komarova2006spatial}, in the context of how cancerous cells spread and metastasize throughout the body of an organism. In \cite{mukhopadhyay2020voter}, \emph{binary opinion dynamics} in a fully connected network of interacting agents is studied, with agents being biased towards one of the two opinions and allowed to interact according to 
\begin{enumerate*}
\item the voter rule, where an updating agent simply copies the opinion of a different, randomly sampled, agent,
\item the majority rule, where an updating agent samples multiple other agents and adopts the opinion that is in majority among the members of this selected sample.
\end{enumerate*}
The diverse applications of the voter model and its variants in the field of social sciences, especially in the analysis of discrete and continuous opinion dynamics, can be seen in \cite{granovsky1995noisy}, \cite{mobilia2003does}, \cite{mobilia2005voting}, \cite{mobilia2007role}, \cite{congleton2004median}, \cite{sood2005voter}, \cite{sood2008voter}, \cite{masuda2010heterogeneous} and \cite{castellano2009nonlinear}, to name just a few. For the sake of completeness, we mention here that the two other extensively studied models of interacting particle systems are the \emph{contact processes} (see, for instance, \cite{griffeath1981basic}, \cite{durrett1988contact}, \cite{bezuidenhout1990critical} and \cite{liggett2013stochastic}) and the \emph{exclusion processes} (see, for instance, \cite{liggett1976coupling}, \cite{de1989weakly}, \cite{quastel1992diffusion}, \cite{sandow1994partially}, \cite{schutz1997exact}, \cite{rajewsky1998asymmetric}, \cite{tracy2008integral}, \cite{mallick2011some} and \cite{liggett2013stochastic}). 

Our learning model is also relevant in the context of \emph{social learning}, a term that was first coined in \cite{ellison1993rules}. In \cite{ellison1993rules}, each economic agent decides between two technologies, whose relative profitability is unknown, by taking into account the experiences of its neighbours and via exogenously specified rules of thumb that ignore historical data. Two learning environments are considered in \cite{ellison1993rules}, in one of which the same technology is considered optimal for \emph{all} agents, whereas in the other, each technology is better for \emph{some} of the agents. In \cite{ellison1995word}, economic agents rely on information obtained via word-of-mouth communication in order to make decisions, without knowing the costs and benefits of the available choices. Two scenarios are considered in \cite{ellison1995word}, the first of which allows a choice between two competing products with unequal qualities or payoffs, while the second allows a choice between two products that are equally good. In \cite{bala1998learning}, when the payoffs from the various possible actions are unknown, agents use their own past experience as well as the experiences of agents belonging to their \emph{neighbourhoods} to guide them in their decision-making, and a general framework to study the relationship between the structure of these neighbourhoods and the process of social learning is proposed and analysed. In \cite{bala2000noncooperative}, an approach to network formation is presented assuming that the link formed by one agent with another allows the former access, in part and in due course, to the benefits available to the latter via their own links, and that the cost of formation of a link is incurred only by the agent who initiates the formation, thereby allowing the network formation process to be formulated as a noncooperative game. In \cite{watts2001dynamic}, the process of network formation in a dynamic framework is analysed, with agents being allowed to form and sever links with one another, while in \cite{mele2017structural}, agents are assumed to meet sequentially at random, myopically updating their links. More about the topic of social learning, as well as about learning models based on networks of agents, can be found in \cite{goyal2012connections}, \cite{goyal2011learning}, \cite{mobius2014social}, \cite{fudenberg2009learning} and \cite{banerjee2004word}, among other resources.

Models for social learning serve as frameworks for the dissemination of technologies throughout a population of agents forming a specified network. In \cite{chatterjee2004technology}, in a population of myopic, memoryless agents stationed at the integer-points of $\mathbb{Z}$, each agent, at each time-step, performs an experiment using one of two technologies available to them, then observes the technology choices and the corresponding outcomes of its two nearest neighbours and of itself. Two learning rules are considered, with the first allowing an agent to change its technology only if it has experienced a failure, and the second allowing each agent to update its technology on the basis of the neighbourhood average. In \cite{bandyopadhyay2010on}, each $x \in \mathbb{Z}$ is occupied by an agent, and each agent, at each time-step, is allowed to adopt one of two available technologies: $B$ and $R$. An agent retains the technology it adopted in the previous time-step if it achieves success using this technology. Otherwise, it observes the technologies adopted, and the outcomes achieved as a consequence, by its two nearest neighbours, and it opts for a change, say, from $B$ to $R$, only if, among its neighbours and itself, the proportion of successes using $R$ is strictly bigger than the proportion of successes using $B$. In \cite{young2009innovation}, models for 
\begin{enumerate*}
\item the spread of a contagion,
\item the spread of social influence,
\item and social learning
\end{enumerate*}
are formulated at a high level of generality, allowing for essentially any distribution of heterogeneous characteristics among the agents concerned, and key differences in their dynamical structures and in their patterns of acceleration are pointed out. In \cite{lamberson2010social}, agents decide whether or not to adopt a new technology, with unknown payoffs, based on their prior beliefs and the experiences of their neighbours in the network, and it is shown, using mean-field approximations, that this process of diffusion always has at least one stable equilibrium. In \cite{sethi2016communication}, each agent
\begin{enumerate*}
\item possesses unobservable perspectives,
\item receives private information about the current state and forms an opinion,
\item and chooses a target agent (different from itself) and observes the target's opinion.
\end{enumerate*}
In \cite{chatterjee2016credibility}, an individual outside the network, termed the \emph{firm}, has privare information regarding the quality of the technology it seeks to diffuse throughout the network, and may pay some agent in the network, known as an \emph{implant}, to propagate its product or idea. Each agent observes the actions of its neighbours over time in order to decide whether or not to adopt the new technology. Agents are either \emph{innovators} who \emph{always} adopt the new technology, or \emph{standard players} who are fully rational and make decisions based on utility maximization. What the firm knows about the interactions of agents with other agents is restricted to the neighbours of the implant chosen, if any exists. In such a model, learning occurs via strategic choices by agents, and this plays a role in determining whether diffusion occurs throughout the whole population or dies out within some finite distance of the origin. Numerous other works exist in the literature on models studying the diffusion of technologies, such as \cite{tsakas2017diffusion}, \cite{namatame2009agent}, \cite{tsakas2024optimal}, \cite{namatame2010diffusion} etc.

A third motivation for studying our model arises from its interpretation as a \emph{probabilistic tree automaton} (PTA). PTAs admit a very general definition, such as has been discussed in \cite{ellis1970probabilistic}, but in this case, it suffices to define a PTA (on $\mathbb{T}_{m}$) as a \emph{state machine} that consists of two components, namely
\begin{enumerate*}
\item a finite set $\mathcal{A}$ of symbols or colours known as the \emph{alphabet}, and
\item a stochastic update rule given by a stochastic matrix $\varphi: \mathcal{A}^{m} \times \mathcal{A} \rightarrow [0,1]$.
\end{enumerate*} 
For any vertex $v \in V(\mathbb{T}_{m})$, letting $v_{1}, v_{2}, \ldots, v_{m}$ denote its children, and conditioned on $\left(\eta_{t}(v_{1}), \eta_{t}(v_{2}), \ldots, \eta_{t}(v_{m})\right)$, where $\eta_{t}(v_{i}) \in \mathcal{A}$ denotes the symbol assigned to $v_{i}$ at time-step $t$ for each $i \in \{1,2,\ldots,m\}$, the symbol assigned to $v$ at time-step $t+1$ is a random variable $\eta_{t+1}(v)$ whose probability distribution is given by
\begin{equation}
\Prob\left[\eta_{t+1}(v)=\alpha\Big|\eta_{t}(v_{1}), \eta_{t}(v_{2}), \ldots, \eta_{t}(v_{m})\right] = \varphi\left(\left(\eta_{t}(v_{1}), \eta_{t}(v_{2}), \ldots, \eta_{t}(v_{m})\right), \alpha\right), \text{ for each } \alpha \in \mathcal{A}.\nonumber
\end{equation}
The update from $\eta_{t}(v)$ to $\eta_{t+1}(v)$ is assumed to happen independently over all $v \in V(\mathbb{T}_{m})$. One of the primary questions asked about a PTA is the enumeration and characterization of all its \emph{fixed points}: a joint law $\nu$ of $\left\{\eta_{0}(v): v \in V(\mathbb{T}_{m})\right\}$ is called a fixed point of the PTA under consideration if the joint law of $\left\{\eta_{1}(v): v \in V(\mathbb{T}_{m})\right\}$ is also $\nu$. This is also the question we seek to answer in this paper for our model. More on PTAs (and their deterministic counterparts, endowed with update rules that are deterministic), as well as their applications, can be found in \cite{cohen2009running}, \cite{bernabeu2011melodic}, \cite{carayol2014randomization} and \cite{comon2008tree}, among others.

\subsection{Organization of this paper}\label{subsec:org} The rest of this paper is organized as follows. The proof of Theorem~\ref{thm:main_general} has been presented in \S\ref{sec:analysis_general}, with \S\ref{subsec:part_1_general}, \S\ref{subsec:part_2_general} and \S\ref{subsec:gen:3_proof} respectively addressing parts \eqref{gen:1}, \eqref{gen:2} and \eqref{gen:3}. The proof of Theorem~\ref{thm:main_m=2} has been outlined in \S\ref{sec:analysis_m=2}. Theorem~\ref{thm:main_equal} has been proved in \S\ref{sec:analysis_equal}, and this section includes, as part of the proof, the statements and proof of the crucial Theorem~\ref{thm:convex_concave}, Lemma~\ref{lem:fixed_point_count_g'(1/2)} and Theorem~\ref{thm:g'_{m}(1/2)_increasing}. It also includes \S\ref{subsec:small_m}, where the value of $p(m)$, as defined in the statement of Theorem~\ref{thm:main_equal}, has been explicitly computed, by analytical means, for $m=3$ and $m=4$. The paper comes to a conclusion with \S\ref{sec:p_{B}=1} which contains the proof of Theorem~\ref{thm:unequal}.

\section{Proof of Theorem~\ref{thm:main_general}}\label{sec:analysis_general}
\subsection{Proof of \eqref{gen:1} of Theorem~\ref{thm:main_general}}\label{subsec:part_1_general} Recall, from \S\ref{subsec:abs_maj}, the description of the model we work with when we consider the absolute majority policy. From \eqref{initial}, we already know that the random variables $C_{0}(v)$, for $v \in V(\mathbb{T}_{m})$, are i.i.d.,\ proving the base case of the inductive argument we plan to outline for the proof of \eqref{gen:1}, the first assertion of Theorem~\ref{thm:main_general}. Let us assume that the random variables $C_{t}(v)$, for $v \in V(\mathbb{T}_{m})$, have been shown to be i.i.d.,\ with their common distribution given by \eqref{time_t}, for each $t \leqslant n$, for some $n \in \mathbb{N}$. We now consider $C_{n+1}(v)$, for $v \in V(\mathbb{T}_{m})$.

Let $S \subset V(\mathbb{T}_{m})$ denote any \emph{finite} subset of vertices of our tree $\mathbb{T}_{m}$. For each $v \in V(\mathbb{T}_{m})$, let $\Gamma(v)$ indicate the set of all children of $v$. Recall, from \S\ref{subsec:abs_maj}, the random variables $X_{t}(v)$ and $Y_{t}(v)$, for each $v \in V(\mathbb{T}_{m})$ and each $t \in \mathbb{N}_{0}$. Since our induction hypothesis states that the collection $\left\{C_{t}(v): v \in V(\mathbb{T}_{m})\right\}$ consists of i.i.d.\ random variables for \emph{each} $t \leqslant n$, and since $\Gamma(v) \cap \Gamma(v') = \emptyset$ for all $v, v' \in V(\mathbb{T}_{m})$ with $v \neq v'$, hence the random tuples $\left(C_{n}(u): u \in \Gamma(v)\right)$ are also i.i.d.\ over \emph{all} $v \in S$. Next, recall from \S\ref{subsec:abs_maj} that, for each $t \in \mathbb{N}_{0}$, the collection $\left\{X_{t}(v): v \in V(\mathbb{T}_{m})\right\} \bigcup \left\{Y_{t}(v): v \in V(\mathbb{T}_{m})\right\}$ is independent. Therefore, the entire collection $\left\{\left(X_{n}(u): u \in \Gamma(v)\right): v \in S\right\} \bigcup \left\{Y_{n}(v): v \in S\right\}$ is independent. 

From \eqref{abs_maj_rule}, it is evident that $C_{n+1}(v)$ is a function of $C_{n}(u)$ and $X_{n}(u)$ for each $u \in \Gamma(v)$, as well as of $Y_{n}(v)$. The discussion in the previous paragraph implies the independence of $\left\{\left(C_{n}(u): u \in \Gamma(v)\right), \left(X_{n}(u): u \in \Gamma(v)\right), Y_{n}(v)\right\}$ over all $v \in S$. This, in turn, implies the independence of $C_{n+1}(v)$ over all $v \in S$. That $C_{n+1}(v)$ has the same distribution for each $v \in V(\mathbb{T}_{m})$ becomes evident from \eqref{abs_maj_rule} (which ensures that the conditional distribution of $C_{n+1}(v)$, conditioned on $C_{n}(u), u \in \Gamma(v)$, is the same for all $v \in V(\mathbb{T}_{m})$) and our induction hypothesis that $\left\{C_{n}(v): v \in V(\mathbb{T}_{m})\right\}$ forms an i.i.d.\ collection.  

\subsection{Proof of \eqref{gen:2} of Theorem~\ref{thm:main_general}}\label{subsec:part_2_general} Having established \eqref{gen:1} of Theorem~\ref{thm:main_general}, let the distribution of $C_{t}(v)$ be as described by \eqref{time_t}, for each $v \in V(\mathbb{T}_{m})$ and each $t \in \mathbb{N}$. Recall that $\nu_{t}$ is the joint law of $\left\{C_{t}(v): v \in V(\mathbb{T}_{m})\right\}$.

We first show that if the sequence $\left\{\pi_{t}\right\}_{t \in \mathbb{N}_{0}}$ converges to a limit $\pi$, then $\left\{\nu_{t}\right\}_{t \in \mathbb{N}_{0}}$ converges to the law $\nu$ of the i.i.d.\ collection of random variables $\left\{C_{\infty}(v): v \in V(\mathbb{T}_{m})\right\}$, where the common distribution of each $C_{\infty}(v)$ is as given in \eqref{limit_pi}. Note that each $\nu_{t}$ is defined on $\{B,R\}^{V(\mathbb{T}_{m})}$, which is compact by Tychonoff's theorem (since the set $\{B,R\}$ is compact), so that the sequence $\left\{\nu_{t}\right\}_{t \in \mathbb{N}_{0}}$ is already tight. Therefore, in order to show that $\left\{\nu_{t}\right\}_{t \in \mathbb{N}_{0}}$ converges to $\nu$, it suffices to show that all finite-dimensional marginals of $\left\{\nu_{t}\right\}_{t \in \mathbb{N}_{0}}$ converge to the corresponding marginals of $\nu$.

To this end, for any $k \in \mathbb{N}$, let $v_{1}, v_{2}, \ldots, v_{k}$ be \emph{any} $k$ vertices in $V(\mathbb{T}_{m})$. Under $\nu_{t}$, the joint law of $\left(C_{t}(v_{1}), C_{t}(v_{2}), \ldots, C_{t}(v_{k})\right)$ is given by
\begin{equation}
\nu_{t}\Big|_{v_{1},v_{2},\ldots,v_{k}}(\eta_{1},\eta_{2},\ldots,\eta_{k}) = \pi_{t}^{\sum_{i=1}^{k}\mathbf{1}_{\eta_{i}=B}}\left(1-\pi_{t}\right)^{k-\sum_{i=1}^{k}\mathbf{1}_{\eta_{i}=B}},\label{lem:1_eq_1}
\end{equation}
for each $(\eta_{1},\eta_{2},\ldots,\eta_{k}) \in \{B,R\}^{k}$. Since $\pi_{t} \rightarrow \pi$ as $t \rightarrow \infty$, and the function in the right side of \eqref{lem:1_eq_1} is a continuous one (in fact, it is a polynomial in $\pi_{t}$), hence 
\begin{equation}
\lim_{t \rightarrow \infty}\pi_{t}^{\sum_{i=1}^{k}\mathbf{1}_{\eta_{i}=B}}\left(1-\pi_{t}\right)^{k-\sum_{i=1}^{k}\mathbf{1}_{\eta_{i}=B}} = \pi^{\sum_{i=1}^{k}\mathbf{1}_{\eta_{i}=B}}\left(1-\pi\right)^{k-\sum_{i=1}^{k}\mathbf{1}_{\eta_{i}=B}} = \nu\Big|_{v_{1},v_{2},\ldots,v_{k}}(\eta_{1},\eta_{2},\ldots,\eta_{k}),\nonumber
\end{equation}
for each $(\eta_{1},\eta_{2},\ldots,\eta_{k}) \in \{B,R\}^{k}$, thus completing the proof of one side of the two-way implication in \eqref{gen:2}.

Conversely, let $\left\{\nu_{t}\right\}_{t \in \mathbb{N}_{0}}$ converge to a distributional limit $\nu$ as $t \rightarrow \infty$. Fix \emph{any} vertex $v$ of $\mathbb{T}_{m}$, and let $C_{\infty}(v)$ denote the random variable that indicates the state of $v$ under the measure $\nu$. Since $C_{\infty}(v)$ takes only the values $B$ and $R$, its distribution must be of the form given in \eqref{limit_pi}, i.e.\ $C_{\infty}(v)=B$ with probability $\pi$ and $C_{\infty}(v)=R$ with probability $1-\pi$. Note that the probability distribution of $C_{\infty}(v)$ is the marginal of the law $\nu$ when we restrict ourselves to the vertex $v$, just as the probability distribution of $C_{t}(v)$ is the marginal of $\nu_{t}$ when we restrict ourselves to $v$. Since $\nu_{t}$ converges to $\nu$ as $t \rightarrow \infty$, we conclude that $C_{t}(v)$ converges in distribution to $C_{\infty}(v)$ as $t \rightarrow \infty$. Since $C_{t}(v)$ is of the form given by \eqref{time_t}, we conclude that $\pi_{t} \rightarrow \pi$ as $t \rightarrow \infty$.

\subsection{Proof of \eqref{gen:3} of Theorem~\ref{thm:main_general}}\label{subsec:gen:3_proof}
We begin by establishing a recurrence relation connecting $\pi_{t}$ with $\pi_{t+1}$, for each $t \in \mathbb{N}_{0}$. Fix $v \in V(\mathbb{T}_{m})$ and let $v_{1}, v_{2}, \ldots, v_{m}$ denote its children. Fix any $k \in \{0,1,\ldots,m\}$. From \eqref{gen:1}, we know that $C_{t}(v_{1})$, $C_{t}(v_{2})$, $\ldots$, $C_{t}(v_{m})$ are i.i.d.\ with each of them following the distribution given by \eqref{time_t}. Hence, the probability that precisely $k$ of these $m$ random variables equal $B$ and the remaining $m-k$ equal $R$ is given by
\begin{equation}
\Prob\left[\sum_{i=1}^{m}\mathbf{1}_{C_{t}(v_{i})=B}=k\right] = {m \choose k}\pi_{t}^{k}(1-\pi_{t})^{m-k}.\label{intermediate_0}
\end{equation}
Conditioned on the event that $C_{t}(v_{i})=B$ for each $i \in \{1,2,\ldots,k\}$, and $C_{t}(v_{j})=R$ for each $j \in \{k+1,\ldots,m\}$, the random variables $X_{t}(v_{1}), \ldots, X_{t}(v_{k})$ are i.i.d.\ Bernoulli$(p_{B})$, while the random variables $X_{t}(v_{k+1}), \ldots, X_{t}(v_{m})$ are i.i.d.\ Bernoulli$(p_{R})$. The absolute majority policy, defined in \eqref{abs_maj_rule}, yields
\begin{align}
{}&\Prob\left[C_{t+1}(v)=B\Big|\text{precisely } k \text{ children of } v \text{ are in state } B \text{ at time-step } t\right] \nonumber\\
={}& \Prob\left[C_{t+1}(v)=B\Big|\sum_{i=1}^{m}\mathbf{1}_{C_{t}(v_{i})=B}=k\right]\nonumber\\
={}&\Prob\left[C_{t+1}(v)=B\Big|C_{t}(v_{i})=B \text{ for } i \in \{1,\ldots,k\}, C_{t}(v_{j})=R \text{ for } j \in \{k+1,\ldots,m\}\right]\nonumber\\
={}&\Prob\left[\sum_{i=1}^{k}X_{t}(v_{i}) > \sum_{j=k+1}^{m}X_{t}(v_{j})\right]+\Prob\left[\sum_{i=1}^{k}X_{t}(v_{i}) = \sum_{j=k+1}^{m}X_{t}(v_{j}), Y_{t}(v)=1\right]\nonumber\\
={}& \Prob\left[A_{k} > B_{m-k}\right]+\frac{1}{2}\Prob\left[A_{k}=B_{m-k}\right],\label{intermediate_1}
\end{align}
where we set $A_{k} = \sum_{i=1}^{k}X_{t}(v_{i})$ and $B_{m-k} = \sum_{j=k+1}^{m}X_{t}(v_{j})$, so that, conditioned on the event mentioned above, $A_{k}$ follows Binomial$(k,p_{B})$ and $B_{m-k}$ follows Binomial$(m-k,p_{R})$. When $k=0$ (respectively, when $k=m$), the random variable $A_{0}$ (respectively, $B_{0}$) is simply degenerate at $0$. The expression obtained in the last step of \eqref{intermediate_1} is precisely $f_{m}(k)$ as defined in \eqref{f_{m}^{abs}}. From \eqref{intermediate_0} and \eqref{intermediate_1}, the unconditional probability of $v$ being in state $B$ at time-step $t+1$ is given by
\begin{align}
\pi_{t+1} ={}& \Prob\left[C_{t+1}(v)=B\right] = \sum_{k=0}^{m}\Prob\left[C_{t+1}(v)=B\Big|\sum_{i=1}^{m}\mathbf{1}_{C_{t}(v_{i})=B}=k\right]\Prob\left[\sum_{i=1}^{m}\mathbf{1}_{C_{t}(v_{i})=B}=k\right]\nonumber\\
={}& \sum_{k=0}^{m}\left\{\Prob\left[A_{k} > B_{m-k}\right]+\frac{1}{2}\Prob\left[A_{k}=B_{m-k}\right]\right\}{m \choose k}\pi_{t}^{k}(1-\pi_{t})^{m-k} = \sum_{k=0}^{m}f_{m}(k){m \choose k}\pi_{t}^{k}(1-\pi_{t})^{m-k},\nonumber
\end{align}
so that, from \eqref{g_{m}^{abs}}, we can write
\begin{equation}
\pi_{t+1} = g_{m}(\pi_{t}) \text{ for each } t \in \mathbb{N}_{0}.\label{intermediate_2}
\end{equation}
When $\pi=\lim_{t \rightarrow \infty}\pi_{t}$ exists, taking the limit, as $t \rightarrow \infty$, on both sides of \eqref{intermediate_2} and using the fact that $g_{m}$ is a polynomial and hence a continuous function, we obtain 
\begin{equation}
\pi = g_{m}(\pi),\label{fixed_point}
\end{equation}
i.e.\ we conclude that $\pi$ is a fixed point of $g_{m}$, in the interval $[0,1]$. This concludes the proof of \eqref{gen:3} of Theorem~\ref{thm:main_general}.

\section{Proof of Theorem~\ref{thm:main_m=2}}\label{sec:analysis_m=2}
The crux of the proof lies in studying the function $g_{2}$. From \eqref{f_{m}^{abs}}, we see that
\begin{align}
{}&f_{2}(0) = \frac{1}{2}\Prob[B_{2}=0] = \frac{1}{2}(1-p_{R})^{2},\nonumber\\
{}&f_{2}(1) = \Prob[A_{1}=1,B_{1}=0]+\frac{1}{2}\Prob[A_{1}=B_{1}=1]+\frac{1}{2}\Prob[A_{1}=B_{1}=0] = p_{B}(1-p_{R})+\frac{p_{B}p_{R}}{2}+\frac{(1-p_{B})(1-p_{R})}{2},\nonumber\\
{}&f_{2}(2) = \Prob[A_{2}\geqslant 1]+\frac{1}{2}\Prob[A_{2}=0] = 1-\frac{1}{2}\Prob[A_{2}=0] = 1-\frac{(1-p_{B})^{2}}{2}.\nonumber
\end{align}
Substituting these in \eqref{g_{m}^{abs}}, we obtain:
\begin{align}
g_{2}(x) ={}& \frac{1}{2}(1-p_{R})^{2}(1-x)^{2}+2\left\{p_{B}(1-p_{R})+\frac{p_{B}p_{R}}{2}+\frac{(1-p_{B})(1-p_{R})}{2}\right\}x(1-x)+\left\{1-\frac{(1-p_{B})^{2}}{2}\right\}x^{2},\nonumber
\end{align}
so that
\begin{align}
\frac{d^{2}}{dx^{2}}g_{2}(x)={}& (1-p_{R})^{2}-4p_{B}(1-p_{R})-2p_{B}p_{R}-2(1-p_{B})(1-p_{R})+2-(1-p_{B})^{2}= p_{R}^{2}-p_{B}^{2},\nonumber
\end{align}
thus implying that $g_{2}$ is strictly convex throughout $[0,1]$ for $p_{R} > p_{B}$, strictly concave throughout $[0,1]$ for $p_{R} < p_{B}$, and linear when $p_{B}=p_{R}$. Furthermore, we have
\begin{equation}
g_{2}(0)=\frac{1}{2}(1-p_{R})^{2} > 0 \text{ for each } p_{R} \in [0,1) \text{ and } g_{2}(1)=1-\frac{(1-p_{B})^{2}}{2}<1 \text{ for all } p_{B} \in [0,1).\nonumber
\end{equation}
Consequently, the curve $y=g_{2}(x)$ lies \emph{above} the line $y=x$ at $x=0$ and \emph{beneath} the line $y=x$ at $x=1$, when $p_{B}, p_{R} \in [0,1)$. The two observations mentioned above, when combined, allow us to conclude that there is a unique point of intersection between $y=g_{2}(x)$ and $y=x$. When $p_{R}=1$ and $p_{B} \in [0,1)$, this unique point of intersection is $x=0$, whereas when $p_{B}=1$ and $p_{R} \in [0,1)$, it is $x=1$. 

Let us now focus on the special case of $p_{B}=p_{R}=p \in [0,1]$. This yields
\begin{equation}
g_{2}(x) = \frac{1}{2}(1-p)^{2}+(2p-p^{2})x,\nonumber
\end{equation}
which, in turn, yields the unique fixed point $x=1/2$ when $p\in [0,1)$, whereas when $p=1$, $y=g_{2}(x)$ coincides with the line $y=x$, and therefore, the entire $[0,1]$ constitutes the set of fixed point of $g_{2}$ in that case.

\section{Proof of Theorem~\ref{thm:main_equal}}\label{sec:analysis_equal}
We begin by summarizing the broad steps via which we accomplish the proof of Theorem~\ref{thm:main_equal} (for each fixed but arbitrary $m \in \mathbb{N}$ with $m \geqslant 2$):
\begin{enumerate}
\item \label{broad_step_1} We begin by observing that $1/2$ is always a fixed point of $g_{m}$, and that $\alpha \in [0,1]$ is a fixed point of $g_{m}$ if and only if $1-\alpha$ is also a fixed point of $g_{m}$.
\item \label{broad_step_2} In Theorem~\ref{thm:convex_concave}, we establish that the function $g_{m}$ is strictly convex on $[0,1/2]$ and strictly concave on $[1/2,1]$ for every $p \in (0,1]$ (for $p=0$, we show that $g_{m}(x)=1/2$ for all $x \in [0,1]$). This, in turn, allows us to conclude that the curve $y=g_{m}(x)$ may intersect the line $y=x$ at most twice in the interval $[0,1/2]$, and at most twice in $[1/2,1]$. Because of \eqref{broad_step_1}, we have two possibilities:
\begin{enumerate}
\item either $g_{m}$ has a unique fixed point, $1/2$, in the interval $[0,1]$,
\item \label{three_fixed_points} or $g_{m}$ has precisely three fixed points in $[0,1]$, and these are of the form $\alpha$, $1/2$ and $1-\alpha$, for some $\alpha \in [0,1/2)$.
\end{enumerate}
\item \label{broad_step_3} We then show, in Lemma~\ref{lem:fixed_point_count_g'(1/2)}, that $g_{m}$ has a unique fixed point in $[0,1]$ if and only if its slope at $1/2$ is bounded above by $1$, i.e.\ $g'_{m}(1/2) \leqslant 1$. 
\item \label{broad_step_4} In Theorem~\ref{thm:g'_{m}(1/2)_increasing}, we show that the derivative $g'_{m}(1/2)$ of $g_{m}$ at $1/2$ is strictly increasing as a function of $p$, for $p \in (0,1)$. This is followed by showing that $g'_{m}(1/2)=0$ at $p=0$, and $g'_{m}(1/2)>1$ at $p=1$. Together, these guarantee the existence of a unique $p(m) \in (0,1)$ such that $g'_{m}(1/2)\leqslant 1$ for all $p \in [0,p(m)]$ and $g'_{m}(1/2)>1$ for all $p \in (p(m),1]$.
\item \label{broad_step_5} From the final conclusion of \eqref{broad_step_4}, from \eqref{broad_step_3} and from \eqref{three_fixed_points}, we conclude that 
\begin{enumerate}
\item $g_{m}$ has a unique fixed point in $[0,1]$ for each $p \in [0,p(m)]$, 
\item and $g_{m}$ has precisely three fixed points, of the form $\alpha$, $1/2$ and $1-\alpha$, for some $\alpha = \alpha(p) < 1/2$, for each $p \in (p(m),1]$.
\end{enumerate}
This completes the proofs of \eqref{thm:equal_part_1} and \eqref{thm:equal_part_2} of Theorem~\ref{thm:main_equal}.
\end{enumerate}

As outlined above, we begin with a couple of simple observations, as follows. When $p_{B}=p_{R}=p$, in the definition of $f_{m}$ in \eqref{f_{m}^{abs}}, the random variable $A_{k}$ follows Binomial$(k,p)$ and the random variable $B_{m-k}$ follows Binomial$(m-k,p)$, for each $k \in \{0,1,\ldots,m\}$. This allows us to write
\begin{align}
f_{m}(m-k) ={}& \Prob\left[A_{m-k} > B_{k}\right] + \frac{1}{2}\Prob\left[A_{m-k}=B_{k}\right]\nonumber\\
={}& 1 - \Prob\left[A_{m-k} < B_{k}\right] - \Prob\left[A_{m-k}=B_{k}\right] + \frac{1}{2}\Prob\left[A_{m-k}=B_{k}\right]\nonumber\\
={}& 1 - \Prob\left[A_{k} > B_{m-k}\right] - \frac{1}{2}\Prob\left[A_{k}=B_{m-k}\right] = 1-f_{m}(k).\label{symm_cond}
\end{align}
In particular, when $m$ is even, this implies that $f_{m}(m/2) = 1/2$. We refer to this condition, which boils down to the identity $f_{m}(k)+f_{m}(m-k)=1$ for each $k \in \{0,1,\ldots,m\}$, as the \emph{symmetry criterion}. Next, using \eqref{symm_cond}, we have, for any $\alpha \in [0,1]$,
\begin{align}
g(\alpha) ={}& \sum_{k=0}^{m}f_{m}(k){m \choose k}\alpha^{k}(1-\alpha)^{m-k} = \sum_{k=0}^{m}\left\{1-f_{m}(m-k)\right\}{m \choose k}\alpha^{k}(1-\alpha)^{m-k} \nonumber\\
={}& \sum_{k=0}^{m}{m \choose k}\alpha^{k}(1-\alpha)^{m-k} - \sum_{k=0}^{m}f_{m}(m-k){m \choose m-k}\alpha^{k}(1-\alpha)^{m-k} = 1-g(1-\alpha),\nonumber
\end{align}
which immediately tells us that when $p_{B}=p_{R}$,
\begin{enumerate}
\item \label{basic_conclusion_1} $1/2$ is \emph{always} a fixed point of $g_{m}$, 
\item \label{basic_conclusion_2} $\alpha \in [0,1]$ is a fixed point of $g_{m}$ if and only if $1-\alpha$ is as well.
\end{enumerate}

This completes \eqref{broad_step_1} and, we now come to \eqref{broad_step_2}. In order to prove Theorem~\ref{thm:main_equal}, it is crucial that we understand what the function $g_{m}$ looks like in the interval $[0,1]$. This is what is captured in Theorem~\ref{thm:convex_concave} below:
\begin{theorem}\label{thm:convex_concave}
For each $p \in (0,1]$, the curve $g_{m}$, defined as in \eqref{g_{m}^{abs}} (with $f_{m}$ as defined in \eqref{f_{m}^{abs}}), is strictly convex on the interval $[0,1/2]$ and strictly concave on the interval $[1/2,1]$.
\end{theorem}
\begin{remark}\label{rem:p=0}
Note that, at $p=0$, $f_{m}(k)=1/2$ for each $k \in \{0,1,\ldots,m\}$, as is evident from \eqref{intermediate_1} (since each of $A_{k}$ and $B_{m-k}$, in \eqref{intermediate_1}, is degenerate at $0$ when $p=0$). Thus \eqref{g_{m}^{abs}} yields $g_{m}(x)=1/2$, a constant function. Therefore, the conclusion of Theorem~\ref{thm:convex_concave} trivially holds for $p=0$. Moreover, this shows that $g_{m}$ has only one fixed point in the interval $[0,1]$, namely $1/2$, when $p=0$.
\end{remark}
\begin{proof}
The crux of the proof lies in determining the sign of $g''(x)$ for $x \in [0,1/2]$ and $x \in [1/2,1]$, separately. To this end, we note, from \eqref{g_{m}^{abs}}, that 
\begin{align}
g'_{m}(x) ={}& \frac{d}{dx}\sum_{k=0}^{m}f_{m}(k){m \choose k}x^{k}(1-x)^{m-k}=m\sum_{\ell=0}^{m-1}\left\{f_{m}(\ell+1)-f_{m}(\ell)\right\}{m-1 \choose \ell}x^{\ell}(1-x)^{m-1-\ell}.\label{g'(x)}
\end{align}
Iterating the same argument that leads to \eqref{g'(x)}, we obtain
\begin{align}
g''_{m}(x) ={}& m(m-1)\sum_{\ell=0}^{m-2}\left\{f_{m}(\ell+2)-2f_{m}(\ell+1)+f_{m}(\ell)\right\}{m-2 \choose \ell}x^{\ell}(1-x)^{m-2-\ell}.\label{g''(x)}
\end{align}
So far, we have not used any property specific to the function $f_{m}$, i.e.\ any property pertaining to the absolute majority policy. In what follows, we make use of \eqref{symm_cond}. The notation $\lceil x \rceil$, for any $x \in \mathbb{R}$, indicates the smallest integer that is greater than or equal to the real number $x$, and the notation $\lfloor x \rfloor$ indicates the largest integer that is bounded above by $x$. When $m$ is even, we note that \eqref{symm_cond} (along with the identity $f_{m}(m/2)=1/2$ that was stated right after \eqref{symm_cond}) yields 
\begin{equation}
f_{m}\left(\frac{m}{2}+1\right)-2f_{m}\left(\frac{m}{2}\right)+f_{m}\left(\frac{m}{2}-1\right) = 1 - 2 \cdot \frac{1}{2} = 0,\nonumber
\end{equation}
so that the summand corresponding to $\ell=\left\lceil\frac{m}{2}\right\rceil-1$ in \eqref{g''(x)} equals $0$ when $m$ is even. This observation, along with the application of \eqref{symm_cond}, allows us to conclude the following:
\begin{align}
g''(x) ={}& m(m-1)\sum_{\ell=0}^{\left\lceil\frac{m}{2}\right\rceil-2}\left\{f_{m}(\ell+2)-2f_{m}(\ell+1)+f_{m}(\ell)\right\}{m-2 \choose \ell}x^{\ell}(1-x)^{m-2-\ell} \nonumber\\&+ \sum_{\ell=\left\lceil\frac{m}{2}\right\rceil-1}^{m-2}\left\{f_{m}(\ell+2)-2f_{m}(\ell+1)+f_{m}(\ell)\right\}{m-2 \choose \ell}x^{\ell}(1-x)^{m-2-\ell}\nonumber\\
={}& m(m-1)\sum_{\ell=0}^{\left\lceil\frac{m}{2}\right\rceil-2}\left\{f_{m}(\ell+2)-2f_{m}(\ell+1)+f_{m}(\ell)\right\}{m-2 \choose \ell}x^{\ell}(1-x)^{\ell}\left\{(1-x)^{m-2-2\ell}-x^{m-2-2\ell}\right\}.\nonumber
\end{align}
The difference $(1-x)^{m-2-2\ell}-x^{m-2-2\ell}$ is strictly positive for $x \in [0,1/2)$, and strictly negative for $x \in (1/2,1]$, for each $\ell \in \left\{0,1,\ldots,\left\lceil\frac{m}{2}\right\rceil-2\right\}$ (this range of values for $\ell$ ensures that the exponent $(m-2-2\ell)$ is strictly positive). To prove Theorem~\ref{thm:convex_concave}, it thus suffices to show that
\begin{equation}
f_{m}(\ell+2)-2f_{m}(\ell+1)+f_{m}(\ell) > 0 \text{ for each } \ell \in \left\{0,1,\ldots,\left\lceil\frac{m}{2}\right\rceil-2\right\} \text{ and each } p \in (0,1].\label{convex_concave_objective}
\end{equation}
The rest of the proof of Theorem~\ref{thm:convex_concave} is dedicated to establishing \eqref{convex_concave_objective}. We first focus on $p \in (0,1)$, and prove \eqref{convex_concave_objective} for the special case of $p=1$ later.

To begin with, we express the left side of the inequality in \eqref{convex_concave_objective} in terms of $X_{t}(v_{1}), X_{t}(v_{2}), \ldots, X_{t}(v_{m})$ (where $v_{1}, v_{2}, \ldots, v_{m}$ are the children of an arbitrary but fixed vertex $v$ of $\mathbb{T}_{m}$), via \eqref{intermediate_1}. We consider the event in which $C_{t}(v_{i})=B$ for each $i \in \{1,2,\ldots,\ell\}$ and $C_{t}(v_{i})=R$ for each $i \in \{\ell+3,\ldots,m\}$ (similar to the event that we conditioned on before beginning the derivation of \eqref{intermediate_1}). The idea we employ is as follows:
\begin{enumerate}
\item when computing $f_{m}(\ell)$, we assume that $C_{t}(v_{\ell+1})=C_{t}(v_{\ell+2})=R$; when computing $f_{m}(\ell+1)$, we assume that $C_{t}(v_{\ell+1})=B$ and $C_{t}(v_{\ell+2})=R$; when computing $f_{m}(\ell+2)$, we assume that $C_{t}(v_{\ell+1})=C_{t}(v_{\ell+2})=B$;
\item we then consider various possible scenarios (in terms of the values of $\sum_{i=1}^{\ell}X_{t}(v_{i})$ and $\sum_{i=\ell+3}^{m}X_{t}(v_{i})$, relative to each other), and determine the extent to which each such scenario contributes to each of the conditional probabilities $f_{m}(\ell)$, $f_{m}(\ell+1)$ and $f_{m}(\ell+2)$;
\item finally, we compute the contribution of each such scenario to $f_{m}(\ell+2)-2f_{m}(\ell+1)+f_{m}(\ell)$, and add everything to obtain our final expression.
\end{enumerate}

\textbf{Scenario 1:} When $\sum_{i=1}^{\ell}X_{t}(v_{i}) > \sum_{i=\ell+3}^{m}X_{t}(v_{i})+2$, we have $\sum_{i=\ell+1}^{m}X_{t}(v_{i}) < \sum_{i=1}^{\ell}X_{t}(v_{i})$ no matter what the values of $X_{t}(v_{\ell+1})$ and $X_{t}(v_{\ell+2})$ are,
which further implies $\sum_{i=\ell+2}^{m}X_{t}(v_{i}) < \sum_{i=1}^{\ell+1}X_{t}(v_{i})$ and $\sum_{i=\ell+3}^{m}X_{t}(v_{i}) < \sum_{i=1}^{\ell+2}X_{t}(v_{i})$. Therefore, the contribution of this scenario to each of $f_{m}(\ell)$, $f_{m}(\ell+1)$ and $f_{m}(\ell+2)$ is $\Prob\left[\sum_{i=1}^{\ell}X_{t}(v_{i}) > \sum_{i=\ell+3}^{m}X_{t}(v_{i})+2\right]$. Consequently, the contribution of this scenario to $f_{m}(\ell+2)-2f_{m}(\ell+1)+f_{m}(\ell)$ equals $0$.

\textbf{Scenario 2:} When $\sum_{i=1}^{\ell}X_{t}(v_{i}) = \sum_{i=\ell+3}^{m}X_{t}(v_{i})+2$, the strict inequality $\sum_{i=\ell+1}^{m}X_{t}(v_{i}) < \sum_{i=1}^{\ell}X_{t}(v_{i})$ holds if and only if $\left(X_{t}(v_{\ell+1}),X_{t}(v_{\ell+2})\right) \in \{(0,0), (0,1), (1,0)\}$, and it becomes an equality if and only if $\left(X_{t}(v_{\ell+1}),X_{t}(v_{\ell+2})\right) = (1,1)$. Thus, the contribution of this scenario to $f_{m}(\ell)$, from \eqref{intermediate_1}, is
\begin{align}
{}&\Prob\left[\sum_{i=1}^{\ell}X_{t}(v_{i}) = \sum_{i=\ell+3}^{m}X_{t}(v_{i})+2,\left(X_{t}(v_{\ell+1}),X_{t}(v_{\ell+2})\right) \in \{(0,0), (0,1), (1,0)\}\right] \nonumber\\&+ \frac{1}{2}\Prob\left[\sum_{i=1}^{\ell}X_{t}(v_{i}) = \sum_{i=\ell+3}^{m}X_{t}(v_{i})+2,\left(X_{t}(v_{\ell+1}),X_{t}(v_{\ell+2})\right) = (1,1)\right].\nonumber
\end{align}
On the other hand, we have $\sum_{i=\ell+2}^{m}X_{t}(v_{i}) < \sum_{i=1}^{\ell+1}X_{t}(v_{i})$ no matter what the values of $X_{t}(v_{\ell+1})$ and $X_{t}(v_{\ell+2})$ are, further implying that $\sum_{i=\ell+3}^{m}X_{t}(v_{i}) < \sum_{i=1}^{\ell+2}X_{t}(v_{i})$. Thus, the contribution of this scenario to each of $f_{m}(\ell+1)$ and $f_{m}(\ell+2)$ is $\Prob\left[\sum_{i=1}^{\ell}X_{t}(v_{i}) = \sum_{i=\ell+3}^{m}X_{t}(v_{i})+2\right]$. The net contribution of this scenario to $f_{m}(\ell+2)-2f_{m}(\ell+1)+f_{m}(\ell)$ equals 
\begin{align}
{}& \Prob\left[\sum_{i=1}^{\ell}X_{t}(v_{i}) = \sum_{i=\ell+3}^{m}X_{t}(v_{i})+2\right] - 2\Prob\left[\sum_{i=1}^{\ell}X_{t}(v_{i}) = \sum_{i=\ell+3}^{m}X_{t}(v_{i})+2\right] \nonumber\\&+ \Prob\left[\sum_{i=1}^{\ell}X_{t}(v_{i}) = \sum_{i=\ell+3}^{m}X_{t}(v_{i})+2,\left(X_{t}(v_{\ell+1}),X_{t}(v_{\ell+2})\right) \in \{(0,0), (0,1), (1,0)\}\right] \nonumber\\&+ \frac{1}{2}\Prob\left[\sum_{i=1}^{\ell}X_{t}(v_{i}) = \sum_{i=\ell+3}^{m}X_{t}(v_{i})+2,\left(X_{t}(v_{\ell+1}),X_{t}(v_{\ell+2})\right) =(1,1)\right]\nonumber\\
={}& -\frac{1}{2}\Prob\left[\sum_{i=1}^{\ell}X_{t}(v_{i}) = \sum_{i=\ell+3}^{m}X_{t}(v_{i})+2,\left(X_{t}(v_{\ell+1}),X_{t}(v_{\ell+2})\right) = (1,1)\right] = -\frac{p^{2}}{2}\Prob\left[\sum_{i=1}^{\ell}X_{t}(v_{i}) = \sum_{i=\ell+3}^{m}X_{t}(v_{i})+2\right].\nonumber
\end{align}

\textbf{Scenario 3:} When $\sum_{i=1}^{\ell}X_{t}(v_{i}) = \sum_{i=\ell+3}^{m}X_{t}(v_{i})+1$, the strict inequality $\sum_{i=\ell+1}^{m}X_{t}(v_{i}) < \sum_{i=1}^{\ell}X_{t}(v_{i})$ holds if and only if $\left(X_{t}(v_{\ell+1}),X_{t}(v_{\ell+2})\right)=(0,0)$, and it becomes an equality if $\left(X_{t}(v_{\ell+1}),X_{t}(v_{\ell+2})\right) \in \{(0,1), (1,0)\}$. Thus, the contribution of this scenario to $f_{m}(\ell)$ is
\begin{align}
{}& \Prob\left[\sum_{i=1}^{\ell}X_{t}(v_{i}) = \sum_{i=\ell+3}^{m}X_{t}(v_{i})+1,\left(X_{t}(v_{\ell+1}),X_{t}(v_{\ell+2})\right)=(0,0)\right] \nonumber\\&+ \frac{1}{2}\Prob\left[\sum_{i=1}^{\ell}X_{t}(v_{i}) = \sum_{i=\ell+3}^{m}X_{t}(v_{i})+1, \left(X_{t}(v_{\ell+1}),X_{t}(v_{\ell+2})\right) \in \{(0,1),(1,0)\}\right].\nonumber
\end{align}
Next, the strict inequality $\sum_{i=\ell+2}^{m}X_{t}(v_{i}) < \sum_{i=1}^{\ell+1}X_{t}(v_{i})$ holds if and only if $\left(X_{t}(v_{\ell+1}),X_{t}(v_{\ell+2})\right) \in \{(0,0),(1,0),(1,1)\}$, and it becomes an equality if and only if $\left(X_{t}(v_{\ell+1}),X_{t}(v_{\ell+2})\right) = (0,1)$. Thus, the contribution of this scenario to $f_{m}(\ell+1)$ is 
\begin{align}
{}&\Prob\left[\sum_{i=1}^{\ell}X_{t}(v_{i}) = \sum_{i=\ell+3}^{m}X_{t}(v_{i})+1,\left(X_{t}(v_{\ell+1}),X_{t}(v_{\ell+2})\right) \in \{(0,0),(1,0),(1,1)\}\right]\nonumber\\&+ \frac{1}{2}\Prob\left[\sum_{i=1}^{\ell}X_{t}(v_{i}) = \sum_{i=\ell+3}^{m}X_{t}(v_{i})+1, \left(X_{t}(v_{\ell+1}),X_{t}(v_{\ell+2})\right)=(0,1)\right].\nonumber
\end{align}
Finally, since $\sum_{i=\ell+3}^{m}X_{t}(v_{i}) < \sum_{i=1}^{\ell+2}X_{t}(v_{i})$ no matter what values $X_{t}(v_{\ell+1})$ and $X_{t}(v_{\ell+2})$ assume, the contribution of this scenario to $f_{m}(\ell+2)$ is thus $\Prob\left[\sum_{i=1}^{\ell}X_{t}(v_{i}) = \sum_{i=\ell+3}^{m}X_{t}(v_{i})+1\right]$. The net contribution of this scenario to $f_{m}(\ell+2)-2f_{m}(\ell+1)+f_{m}(\ell)$ equals
\begin{align} 
{}&\Prob\left[\sum_{i=1}^{\ell}X_{t}(v_{i}) = \sum_{i=\ell+3}^{m}X_{t}(v_{i})+1,\left(X_{t}(v_{\ell+1}),X_{t}(v_{\ell+2})\right)=(0,0)\right] \nonumber\\&+ \frac{1}{2}\Prob\left[\sum_{i=1}^{\ell}X_{t}(v_{i}) = \sum_{i=\ell+3}^{m}X_{t}(v_{i})+1, \left(X_{t}(v_{\ell+1}),X_{t}(v_{\ell+2})\right) \in \{(0,1),(1,0)\}\right]\nonumber\\&-2\Prob\left[\sum_{i=1}^{\ell}X_{t}(v_{i}) = \sum_{i=\ell+3}^{m}X_{t}(v_{i})+1,\left(X_{t}(v_{\ell+1}),X_{t}(v_{\ell+2})\right) \in \{(0,0),(1,0),(1,1)\}\right]\nonumber\\& - \Prob\left[\sum_{i=1}^{\ell}X_{t}(v_{i}) = \sum_{i=\ell+3}^{m}X_{t}(v_{i})+1,\left(X_{t}(v_{\ell+1}),X_{t}(v_{\ell+2})\right)=(0,1)\right]+\Prob\left[\sum_{i=1}^{\ell}X_{t}(v_{i}) = \sum_{i=\ell+3}^{m}X_{t}(v_{i})+1\right]\nonumber\\
={}& -\frac{1}{2}\Prob\left[\sum_{i=1}^{\ell}X_{t}(v_{i}) = \sum_{i=\ell+3}^{m}X_{t}(v_{i})+1,\left(X_{t}(v_{\ell+1}),X_{t}(v_{\ell+2})\right)=(1,0)\right]\nonumber\\&+\frac{1}{2}\Prob\left[\sum_{i=1}^{\ell}X_{t}(v_{i}) = \sum_{i=\ell+3}^{m}X_{t}(v_{i})+1,\left(X_{t}(v_{\ell+1}),X_{t}(v_{\ell+2})\right)=(0,1)\right]\nonumber\\&-\Prob\left[\sum_{i=1}^{\ell}X_{t}(v_{i}) = \sum_{i=\ell+3}^{m}X_{t}(v_{i})+1,\left(X_{t}(v_{\ell+1}),X_{t}(v_{\ell+2})\right)=(1,1)\right] = -p^{2}\Prob\left[\sum_{i=1}^{\ell}X_{t}(v_{i}) = \sum_{i=\ell+3}^{m}X_{t}(v_{i})+1\right].\nonumber
\end{align}

\textbf{Scenario 4:} When $\sum_{i=1}^{\ell}X_{t}(v_{i})=\sum_{i=\ell+3}^{m}X_{t}(v_{i})$, the strict inequality $\sum_{i=\ell+1}^{m}X_{t}(v_{i}) < \sum_{i=1}^{\ell}X_{t}(v_{i})$ never holds, but equality does if and only if $\left(X_{t}(v_{\ell+1}),X_{t}(v_{\ell+2})\right)=(0,0)$. The contribution of this scenario to $f_{m}(\ell)$ is thus
\begin{equation}
\frac{1}{2}\Prob\left[\sum_{i=1}^{\ell}X_{t}(v_{i})=\sum_{i=\ell+3}^{m}X_{t}(v_{i}),\left(X_{t}(v_{\ell+1}),X_{t}(v_{\ell+2})\right)=(0,0)\right].\nonumber
\end{equation}
Next, the strict inequality $\sum_{i=\ell+2}^{m}X_{t}(v_{i})<\sum_{i=1}^{\ell+1}X_{t}(v_{i})$ holds if and only if $\left(X_{t}(v_{\ell+1},X_{t}(v_{\ell+2})\right)=(1,0)$, and it becomes an equality if and only if $\left(X_{t}(v_{\ell+1},X_{t}(v_{\ell+2})\right) \in \{(0,0),(1,1)\}$. The contribution of this scenario to $f_{m}(\ell+1)$ is thus
\begin{align}
{}&\Prob\left[\sum_{i=1}^{\ell}X_{t}(v_{i})=\sum_{i=\ell+3}^{m}X_{t}(v_{i}),\left(X_{t}(v_{\ell+1}),X_{t}(v_{\ell+2})\right)=(1,0)\right]\nonumber\\&+\frac{1}{2}\Prob\left[\sum_{i=1}^{\ell}X_{t}(v_{i})=\sum_{i=\ell+3}^{m}X_{t}(v_{i}),\left(X_{t}(v_{\ell+1}),X_{t}(v_{\ell+2})\right)\in \{(0,0),(1,1)\}\right].\nonumber
\end{align} 
Finally, the strict inequality $\sum_{i=\ell+3}^{m}X_{t}(v_{i})<\sum_{i=1}^{\ell+2}X_{t}(v_{i})$ holds if and only if $\left(X_{t}(v_{\ell+1},X_{t}(v_{\ell+2})\right) \in \{(0,1),(1,0),(1,1)\}$, and it becomes an equality if and only if $\left(X_{t}(v_{\ell+1},X_{t}(v_{\ell+2})\right) = (0,0)$. Therefore, the contribution of this scenario to $f_{m}(\ell+2)$ is 
\begin{align}
{}&\Prob\left[\sum_{i=1}^{\ell}X_{t}(v_{i})=\sum_{i=\ell+3}^{m}X_{t}(v_{i}),\left(X_{t}(v_{\ell+1}),X_{t}(v_{\ell+2})\right) \in \{(0,1),(1,0),(1,1)\}\right]\nonumber\\&+\frac{1}{2}\Prob\left[\sum_{i=1}^{\ell}X_{t}(v_{i})=\sum_{i=\ell+3}^{m}X_{t}(v_{i}),\left(X_{t}(v_{\ell+1}),X_{t}(v_{\ell+2})\right)=(0,0)\right].\nonumber
\end{align}
The net contribution of this scenario to $f_{m}(\ell+2)-2f_{m}(\ell+1)+f_{m}(\ell)$ is thus 
\begin{align}
{}&\frac{1}{2}\Prob\left[\sum_{i=1}^{\ell}X_{t}(v_{i})=\sum_{i=\ell+3}^{m}X_{t}(v_{i}),\left(X_{t}(v_{\ell+1}),X_{t}(v_{\ell+2})\right)=(0,0)\right]\nonumber\\&-2\Prob\left[\sum_{i=1}^{\ell}X_{t}(v_{i})=\sum_{i=\ell+3}^{m}X_{t}(v_{i}),\left(X_{t}(v_{\ell+1}),X_{t}(v_{\ell+2})\right)=(1,0)\right]\nonumber\\&-\Prob\left[\sum_{i=1}^{\ell}X_{t}(v_{i})=\sum_{i=\ell+3}^{m}X_{t}(v_{i}),\left(X_{t}(v_{\ell+1}),X_{t}(v_{\ell+2})\right)\in \{(0,0),(1,1)\}\right]\nonumber\\&+\Prob\left[\sum_{i=1}^{\ell}X_{t}(v_{i})=\sum_{i=\ell+3}^{m}X_{t}(v_{i}),\left(X_{t}(v_{\ell+1}),X_{t}(v_{\ell+2})\right) \in \{(0,1),(1,0),(1,1)\}\right]\nonumber\\&+\frac{1}{2}\Prob\left[\sum_{i=1}^{\ell}X_{t}(v_{i})=\sum_{i=\ell+3}^{m}X_{t}(v_{i}),\left(X_{t}(v_{\ell+1}),X_{t}(v_{\ell+2})\right)=(0,0)\right]\nonumber\\
={}& -\Prob\left[\sum_{i=1}^{\ell}X_{t}(v_{i})=\sum_{i=\ell+3}^{m}X_{t}(v_{i}),\left(X_{t}(v_{\ell+1}),X_{t}(v_{\ell+2})\right)=(1,0)\right] \nonumber\\&+ \Prob\left[\sum_{i=1}^{\ell}X_{t}(v_{i})=\sum_{i=\ell+3}^{m}X_{t}(v_{i}),\left(X_{t}(v_{\ell+1}),X_{t}(v_{\ell+2})\right)=(0,1)\right] = 0.\nonumber
\end{align}

\textbf{Scenario 5:} When $\sum_{i=1}^{\ell}X_{t}(v_{i})=\sum_{i=\ell+3}^{m}X_{t}(v_{i})-1$, we have $\sum_{i=1}^{\ell}X_{t}(v_{i}) < \sum_{i=\ell+1}^{m}X_{t}(v_{i})$ no matter what the values of $X_{t}(v_{\ell+1})$ and $X_{t}(v_{\ell+2})$ are, so that this scenario contributes nothing to $f_{m}(\ell)$. Next, we have $\sum_{i=\ell+2}^{m}X_{t}(v_{i}) = \sum_{i=1}^{\ell+1}X_{t}(v_{i})$ if and only if $\left(X_{t}(v_{\ell+1}),X_{t}(v_{\ell+2})\right)=(1,0)$, and under no circumstances do we have the strict inequality $\sum_{i=\ell+2}^{m}X_{t}(v_{i}) < \sum_{i=1}^{\ell+1}X_{t}(v_{i})$. Thus, the contribution of this scenario to $f_{m}(\ell+1)$ is 
\begin{equation}
\frac{1}{2}\Prob\left[\sum_{i=1}^{\ell}X_{t}(v_{i})=\sum_{i=\ell+3}^{m}X_{t}(v_{i})-1,\left(X_{t}(v_{\ell+1}),X_{t}(v_{\ell+2})\right)=(1,0)\right].\nonumber
\end{equation}
Finally, the strict inequality $\sum_{i=\ell+3}^{m}X_{t}(v_{i}) < \sum_{i=1}^{\ell+2}X_{t}(v_{i})$ holds if and only if $\left(X_{t}(v_{\ell+1}),X_{t}(v_{\ell+2})\right)=(1,1)$, and this becomes an equality if and only if $\left(X_{t}(v_{\ell+1}),X_{t}(v_{\ell+2})\right) \in \{(0,1),(1,0)\}$. Thus, the contribution of this scenario to $f_{m}(\ell+2)$ equals
\begin{align}
{}&\Prob\left[\sum_{i=1}^{\ell}X_{t}(v_{i})=\sum_{i=\ell+3}^{m}X_{t}(v_{i})-1,\left(X_{t}(v_{\ell+1}),X_{t}(v_{\ell+2})\right)=(1,1)\right]\nonumber\\&+\frac{1}{2}\Prob\left[\sum_{i=1}^{\ell}X_{t}(v_{i})=\sum_{i=\ell+3}^{m}X_{t}(v_{i})-1,\left(X_{t}(v_{\ell+1}),X_{t}(v_{\ell+2})\right)\in \{(0,1),(1,0)\}\right].\nonumber
\end{align}
The net contribution of this scenario to $f_{m}(\ell+2)-2f_{m}(\ell+1)+f_{m}(\ell)$ is thus
\begin{align}
{}&\Prob\left[\sum_{i=1}^{\ell}X_{t}(v_{i})=\sum_{i=\ell+3}^{m}X_{t}(v_{i})-1,\left(X_{t}(v_{\ell+1}),X_{t}(v_{\ell+2})\right)=(1,1)\right]\nonumber\\&+\frac{1}{2}\Prob\left[\sum_{i=1}^{\ell}X_{t}(v_{i})=\sum_{i=\ell+3}^{m}X_{t}(v_{i})-1,\left(X_{t}(v_{\ell+1}),X_{t}(v_{\ell+2})\right)\in \{(0,1),(1,0)\}\right]\nonumber\\&-\Prob\left[\sum_{i=1}^{\ell}X_{t}(v_{i})=\sum_{i=\ell+3}^{m}X_{t}(v_{i})-1,\left(X_{t}(v_{\ell+1}),X_{t}(v_{\ell+2})\right)=(1,0)\right]=p^{2}\Prob\left[\sum_{i=1}^{\ell}X_{t}(v_{i})=\sum_{i=\ell+3}^{m}X_{t}(v_{i})-1\right].\nonumber
\end{align}

\textbf{Scenario 6:} When $\sum_{i=1}^{\ell}X_{t}(v_{i})=\sum_{i=\ell+3}^{m}X_{t}(v_{i})-2$, we have $\sum_{i=\ell+2}^{m}X_{t}(v_{i})=\sum_{i=\ell+3}^{m}X_{t}(v_{i})+X_{t}(v_{\ell+2})=\sum_{i=1}^{\ell}X_{t}(v_{i})+2+X_{t}(v_{\ell+2}) > \sum_{i=1}^{\ell}X_{t}(v_{i})+X_{t}(v_{\ell+1}) = \sum_{i=1}^{\ell+1}X_{t}(v_{i})$ no matter what the values of $X_{t}(v_{\ell+1})$ and $X_{t}(v_{\ell+2})$ are. This inequality further implies that $\sum_{i=\ell+1}^{m}X_{t}(v_{i}) > \sum_{i=1}^{\ell}X_{t}(v_{i})$. Thus, this scenario leaves no contribution in $f_{m}(\ell+1)$ nor in $f_{m}(\ell)$. We have $\sum_{i=\ell+3}^{m}X_{t}(v_{i})=\sum_{i=1}^{\ell+2}X_{t}(v_{i})$ if and only if $\left(X_{t}(v_{\ell+1}),X_{t}(v_{\ell+2})\right)=(1,1)$. Thus, the contribution of this scenario to $f_{m}(\ell+2)$, and in fact, to $f_{m}(\ell+2)-2f_{m}(\ell+1)+f_{m}(\ell)$, is 
\begin{equation}
\frac{1}{2}\Prob\left[\sum_{i=1}^{\ell}X_{t}(v_{i})=\sum_{i=\ell+3}^{m}X_{t}(v_{i})-2,\left(X_{t}(v_{\ell+1}),X_{t}(v_{\ell+2})\right)=(1,1)\right] = \frac{p^{2}}{2}\Prob\left[\sum_{i=1}^{\ell}X_{t}(v_{i})=\sum_{i=\ell+3}^{m}X_{t}(v_{i})-2\right].\nonumber
\end{equation}

\textbf{Scenario 7:} When $\sum_{i=1}^{\ell}X_{t}(v_{i})<\sum_{i=\ell+3}^{m}X_{t}(v_{i})-2$, it is straightforward to see that it contributes to none of $f_{m}(\ell)$, $f_{m}(\ell+1)$ and $f_{m}(\ell+2)$, so that its net contribution to $f_{m}(\ell+2)-2f_{m}(\ell+1)+f_{m}(\ell)$ is $0$. 

Combining the contributions obtained from all the scenarios considered above, we conclude that 
\begin{align}
{}&f_{m}(\ell+2)-2f_{m}(\ell+1)+f_{m}(\ell) \nonumber\\
={}& p^{2}\left\{\Prob\left[\sum_{i=1}^{\ell}X_{t}(v_{i})=\sum_{i=\ell+3}^{m}X_{t}(v_{i})-1\right]-\Prob\left[\sum_{i=1}^{\ell}X_{t}(v_{i}) = \sum_{i=\ell+3}^{m}X_{t}(v_{i})+1\right]\right\}\nonumber\\&+ \frac{p^{2}}{2}\left\{\Prob\left[\sum_{i=1}^{\ell}X_{t}(v_{i})=\sum_{i=\ell+3}^{m}X_{t}(v_{i})-2\right]-\Prob\left[\sum_{i=1}^{\ell}X_{t}(v_{i})=\sum_{i=\ell+3}^{m}X_{t}(v_{i})+2\right]\right\}.\label{intermediate_3}
\end{align}
Since $\sum_{i=1}^{\ell}X_{t}(v_{i})$ follows Binomial$(\ell,p)$, $\sum_{i=\ell+3}^{m}X_{t}(v_{i})$ follows Binomial$(m-\ell-2,p)$, and they are independent of each other, we have (the strict inequality holds because $p \in (0,1)$, making the omitted term strictly positive):
\begin{align}
{}& \Prob\left[\sum_{i=1}^{\ell}X_{t}(v_{i})=\sum_{i=\ell+3}^{m}X_{t}(v_{i})-1\right]-\Prob\left[\sum_{i=1}^{\ell}X_{t}(v_{i})=\sum_{i=\ell+3}^{m}X_{t}(v_{i})+1\right]\nonumber\\
={}& \sum_{r=0}^{\ell}\Prob\left[\sum_{i=1}^{\ell}X_{t}(v_{i})=r,\sum_{i=\ell+3}^{m}X_{t}(v_{i})=r+1\right]-\sum_{r=0}^{\ell-1}\Prob\left[\sum_{i=1}^{\ell}X_{t}(v_{i})=r+1,\sum_{i=\ell+3}^{m}X_{t}(v_{i})=r\right]\nonumber\\
> {}& \sum_{r=0}^{\ell-1}{\ell \choose r}p^{r}(1-p)^{\ell-r}{m-2-\ell \choose r+1}p^{r+1}(1-p)^{m-3-\ell-r}\nonumber\\&-\sum_{r=0}^{\ell-1}{\ell \choose r+1}p^{r+1}(1-p)^{\ell-r-1}{m-2-\ell \choose r}p^{r}(1-p)^{m-2-\ell-r}\nonumber\\
={}& \sum_{r=0}^{\ell-1}\frac{\ell!(m-2-\ell)!}{r!(r+1)!(\ell-r)!(m-2-r-\ell)!}\left\{(m-2-r-\ell)-(\ell-r)\right\}p^{2r+1}(1-p)^{m-3-2r}\nonumber\\
={}& \sum_{r=0}^{\ell-1}\frac{\ell!(m-2-\ell)!(m-2-2\ell)}{r!(r+1)!(\ell-r)!(m-2-r-\ell)!}p^{2r+1}(1-p)^{m-3-2r},\label{intermediate_4}
\end{align}
and this is strictly positive for all $p \in (0,1)$ since $\ell \leqslant \left\lceil\frac{m}{2}\right\rceil-2 \implies \implies m-2-2\ell \geqslant 1$. This takes care of the first difference in the expression in \eqref{intermediate_3}.

We deal with the second difference in \eqref{intermediate_3} in much the same manner as the first, but with a couple of additional observations, as follows: 
\begin{enumerate}
\item When $m$ is odd and $\ell = \left\lceil\frac{m}{2}\right\rceil-2$, we have $\ell+2 > m-\ell-2 = \ell+1$, which means that the event $\left\{\sum_{i=1}^{\ell}X_{t}(v_{i})=\sum_{i=\ell+3}^{m}X_{t}(v_{i})-2\right\}$ equals the union $\bigcup_{r=0}^{\ell-1}\left\{\sum_{i=1}^{\ell}X_{t}(v_{i})=r,\sum_{i=\ell+3}^{m}X_{t}(v_{i})=r+2\right\}$, whereas when $\ell < \left\lceil\frac{m}{2}\right\rceil-2$, it equals the union $\bigcup_{r=0}^{\ell}\left\{\sum_{i=1}^{\ell}X_{t}(v_{i})=r,\sum_{i=\ell+3}^{m}X_{t}(v_{i})=r+2\right\}$. 
\item When $m$ is even, the event $\left\{\sum_{i=1}^{\ell}X_{t}(v_{i})=\sum_{i=\ell+3}^{m}X_{t}(v_{i})-2\right\}$ equals the union $\bigcup_{r=0}^{\ell}\left\{\sum_{i=1}^{\ell}X_{t}(v_{i})=r,\sum_{i=\ell+3}^{m}X_{t}(v_{i})=r+2\right\}$ for \emph{each} $\ell \in \left\{0,1,\ldots,\left\lceil\frac{m}{2}\right\rceil-2\right\}$. 
\end{enumerate}
We can thus write $\Prob\left[\sum_{i=1}^{\ell}X_{t}(v_{i})=\sum_{i=\ell+3}^{m}X_{t}(v_{i})-2\right] \geqslant \sum_{r=0}^{\ell-1}\Prob\left[\sum_{i=1}^{\ell}X_{t}(v_{i})=r,\sum_{i=\ell+3}^{m}X_{t}(v_{i})=r+2\right]$. Therefore, the second difference in the expression of \eqref{intermediate_3} becomes (once again, the strict inequality holds because $p \in (0,1)$, making the omitted term strictly positive):
\begin{align}
{}&\Prob\left[\sum_{i=1}^{\ell}X_{t}(v_{i})=\sum_{i=\ell+3}^{m}X_{t}(v_{i})-2\right]-\Prob\left[\sum_{i=1}^{\ell}X_{t}(v_{i})=\sum_{i=\ell+3}^{m}X_{t}(v_{i})+2\right]\nonumber\\
\geqslant{}& \sum_{r=0}^{\ell-1}\Prob\left[\sum_{i=1}^{\ell}X_{t}(v_{i})=r,\sum_{i=\ell+3}^{m}X_{t}(v_{i})=r+2\right] - \sum_{r=0}^{\ell-2}\Prob\left[\sum_{i=1}^{\ell}X_{t}(v_{i})=r+2,\sum_{i=\ell+3}^{m}X_{t}(v_{i})=r\right]\nonumber\\
>{}& \sum_{r=0}^{\ell-2}{\ell \choose r}p^{r}(1-p)^{\ell-r}{m-\ell-2 \choose r+2}p^{r+2}(1-p)^{m-\ell-r-4}\nonumber\\&- \sum_{r=0}^{\ell-2}{\ell \choose r+2}p^{r+2}(1-p)^{\ell-r-2}{m-\ell-2 \choose r}p^{r}(1-p)^{m-\ell-r-2}\nonumber\\
={}& \sum_{r=0}^{\ell-2}\frac{\ell!(m-2-\ell)!\left\{(m-\ell-r-3)(m-\ell-r-2)-(\ell-r-1)(\ell-r)\right\}}{r!(r+2)!(\ell-r)!(m-\ell-r-2)!}p^{2r+2}(1-p)^{m-2r-4},\label{intermediate_5}
\end{align}
and this is strictly positive for $p \in (0,1)$ because, as noted above, $\ell \leqslant \left\lceil\frac{m}{2}\right\rceil-2 \implies m-2-2\ell \geqslant 1$, which in turn implies that $m-\ell-r-2 > \ell-r$ as well as $m-\ell-r-3 > \ell-r-1$.

Since we have established that the expression in \eqref{intermediate_4} as well as that in \eqref{intermediate_5} is strictly positive for each $\ell \in \left\{0,1,\ldots,\left\lceil\frac{m}{2}\right\rceil-2\right\}$ and for each $p \in (0,1)$, so is the expression in \eqref{intermediate_3}, which in turn accomplishes our objective of proving \eqref{convex_concave_objective} for $p \in (0,1)$.

We now come to the proof of \eqref{convex_concave_objective} for $p=1$. Note, in this case, from \eqref{intermediate_1}, that
\begin{enumerate*}
\item $f_{m}(\ell)=1$ if $\ell > m-\ell$, 
\item $f_{m}(\ell)=1/2$ if $\ell=m-\ell$ (which only happens if $m$ is even),
\item and $f_{m}(\ell)=0$ if $\ell < m-\ell$.
\end{enumerate*}
Note that $\ell \leqslant \left\lceil\frac{m}{2}\right\rceil-2 \implies \ell+1 \leqslant m-\ell-2 < m-(\ell+1)$, so that $f_{m}(\ell+1)=0$, further implying $f_{m}(\ell)=0$, for each $\ell \in \left\{0,1,\ldots,\left\lceil\frac{m}{2}\right\rceil-2\right\}$. Moreover, $\ell < \left\lceil\frac{m}{2}\right\rceil-2 \implies \ell \leqslant \left\lceil\frac{m}{2}\right\rceil-3 \implies \ell+2 \leqslant m-\ell-3 < m-(\ell+2)$, so that $f_{m}(\ell+2)=0$ for all $\ell < \left\lceil\frac{m}{2}\right\rceil-2$. When $m=2n$ for some $n \in \mathbb{N}$, and $\ell=\left\lceil\frac{m}{2}\right\rceil-2=n-2$, we have $\ell+2=n=m-\ell-2$, implying that $f_{m}(\ell+2)=1/2$ in this case. When $m=2n+1$ for some $n \in \mathbb{N}$, and $\ell=\left\lceil\frac{m}{2}\right\rceil-2=n-1$, we have $\ell+2=n+1>n=m-\ell-2$, so that $f_{m}(\ell+2)=1$ in this case. All these observations together imply that the expression $f_{m}(\ell+2)-2f_{m}(\ell+1)+f_{m}(\ell)$ equals $0$ for each $\ell < \left\lceil\frac{m}{2}\right\rceil-2$, and when $\ell=\left\lceil\frac{m}{2}\right\rceil-2$, it equals $1/2$ for $m$ even and $1$ for $m$ odd, once again establishing \eqref{convex_concave_objective}.

This completes the proof of \eqref{convex_concave_objective}, and consequently, Theorem~\ref{thm:convex_concave}, for all $p \in (0,1]$.
\end{proof}

Before we proceed any further, we explain why we already have the full picture of $g_{m}$ when $p=1$. From Theorem~\ref{thm:convex_concave}, we conclude that the curve $y=g_{m}(x)$ has at most two points of intersection with the line $y=x$ in \emph{each} of the two intervals $[0,1/2]$ and $[1/2,1]$. From \eqref{basic_conclusion_1}, we know that $1/2$ is already a fixed point, which means that $g_{m}$ is allowed to have at most one fixed point in the interval $[0,1/2)$, and at most one fixed point in the interval $(1/2,1]$. From \eqref{intermediate_1}, we note that 
\begin{enumerate*}
\item $f_{m}(k)=1$ for each $k \in \{n+1,\ldots,2n+1\}$ when $m=2n+1$,
\item $f_{m}(k)=1$ for each $k \in \{n+1,\ldots,2n\}$ and $f_{m}(n)=1/2$, when $m=2n$,
\end{enumerate*}
when $p=1$. This tells us that, when $p=1$, we have $g_{m}(0)=0$, and this, along with \eqref{basic_conclusion_2}, yields \emph{both} $0$ and $1$ as fixed points of $g_{m}$ when $p=1$. Consequently, $0$, $1/2$ and $1$ are the three fixed points of $g_{m}$ when $p=1$.

Recall from Remark~\ref{rem:p=0} that $y=g_{m}(x)$ coincides with the line $y=1/2$ when $p=0$, so that $g_{m}$ has the unique fixed point $1/2$ in $[0,1]$ when $p=0$. We, therefore, need only focus on investigating the number of fixed points of $g_{m}$ when $p \in (0,1)$. The following result is the first step towards such an investigation:
\begin{lemma}\label{lem:fixed_point_count_g'(1/2)}
For $p \in (0,1)$, the function $g_{m}$ has a unique fixed point in $[0,1]$ if and only if its derivative $g'_{m}(1/2)$ at $1/2$ is bounded above by $1$.
\end{lemma}  
\begin{proof}
Note that, from \eqref{g_{m}^{abs}} and \eqref{intermediate_1}, we have $g_{m}(0)=f_{m}(0)=\frac{1}{2}\Prob[B_{m}=0] = \frac{1}{2}(1-p)^{m} > 0$ (since $p < 1$), so that the curve $y=g_{m}(x)$ lies \emph{above} the line $y=x$ at $x=0$. If $\alpha$ is the smallest fixed point of $g_{m}$ in the interval $[0,1]$ (and $\alpha > 0$ since we have just shown that $g_{m}(0) > 0$), the curve $y=g_{m}(x)$ travels from \emph{above} $y=x$ to \emph{beneath} $y=x$ at $x=\alpha$. Therefore, the slope $g'_{m}(\alpha)$ of the curve $y=g_{m}(x)$ at $x=\alpha$ must be bounded above by the slope of the line $y=x$, which is $1$. We already know, from \eqref{basic_conclusion_1}, that $1/2$ is a fixed point of $g_{m}$. When $g_{m}$ has a unique fixed point in $[0,1]$, it must be $\alpha=1/2$, allowing us to conclude that $g'_{m}(1/2) \leqslant 1$. 

When $g_{m}$ has multiple fixed points in $[0,1]$, we know, from \eqref{basic_conclusion_1} and \eqref{basic_conclusion_2}, that $\alpha < 1/2$, and that all three of $\alpha$, $1/2$ and $1-\alpha$ are fixed points of $g_{m}$. From Theorem~\ref{thm:convex_concave}, we know that $g_{m}$ has at most two fixed points in $[0,1/2]$ and at most two in $[1/2,1]$, so that $\alpha$, $1/2$ and $1-\alpha$ constitute \emph{all} of the fixed points of $g_{m}$ in $[0,1]$. Evidently, the curve $y=g_{m}(x)$ travels from \emph{above} $y=x$ to \emph{beneath} $y=x$ at $x=\alpha$, whereas it travels from \emph{beneath} $y=x$ to \emph{above} $y=x$ at $x=1/2$. Consequently, the slope $g'_{m}(1/2)$ of $g_{m}$ at $1/2$ must be strictly greater than the slope of the line $y=x$, which is $1$, yielding $g'_{m}(1/2) > 1$ in this case, as desired. This completes the proof.
\end{proof}

We now come to the final step, i.e.\ \eqref{broad_step_4}, of proving Theorem~\ref{thm:main_equal}. From Lemma~\ref{lem:fixed_point_count_g'(1/2)} as well as the second paragraph of its proof, it suffices for us to show that, for each $m \in \mathbb{N}$ with $m \geqslant 2$, there exists $p(m) \in (0,1)$ such that $g'_{m}(1/2) \leqslant 1$ for each $p \leqslant p(m)$, and $g'_{m}(1/2) > 1$ for each $p > p(m)$. This is precisely what is accomplished by proving the following result:
\begin{theorem}\label{thm:g'_{m}(1/2)_increasing}
For each $m \in \mathbb{N}$ with $m \geqslant 2$, the derivative $g'_{m}(1/2)$ is a strictly increasing function of the parameter $p$, for $p \in (0,1)$.
\end{theorem}
\begin{proof}
We have, so far, used the notation $f_{m}(\ell)$, for each $\ell \in \{0,1,\ldots,m-1\}$, that does \emph{not} make explicit the dependence that these functions (recall their definitions from \eqref{f_{m}^{abs}}) have on $p$. Since this proof constitutes investigating the behaviour of these functions with respect to $p$, we tweak these notations a little (just for the sake of this proof and not beyond that), and replace $f_{m}(\ell)$ by $f_{m}^{\ell}(p)$, for each $\ell \in \left\{0,1,\ldots,m\right\}$. Likewise, to emphasize the dependence of $g'_{m}(1/2)$ on the value of $p$ under consideration, we write $g'_{m}(1/2)\big|_{p}$ throughout the proof of Theorem~\ref{thm:g'_{m}(1/2)_increasing}.

Utilizing \eqref{g'(x)} and \eqref{symm_cond} (along with the fact that $f_{m}^{m/2}(p)=1/2$ when $m$ is even, stated right after \eqref{symm_cond}) to obtain, for each $p \in (0,1)$:
\begin{align}
g'_{m}\left(\frac{1}{2}\right)\Big|_{p} ={}& \frac{m}{2^{m-1}}\sum_{\ell=0}^{m-1}\left\{f_{m}^{\ell+1}(p)-f_{m}^{\ell}(p)\right\}{m-1 \choose \ell}=\frac{m}{2^{m-2}}\left[\sum_{\ell=0}^{\left\lfloor\frac{m-1}{2}\right\rfloor}\frac{(m-1)!(2\ell-m)f_{m}^{\ell}(p)}{\ell!(m-\ell)!}+{m-1 \choose \left\lfloor\frac{m-1}{2}\right\rfloor}\right].\label{g'_{m}(1/2)}
\end{align}
Our task is to establish that the derivative of the expression in \eqref{g'_{m}(1/2)}, with respect to $p$, is strictly positive for each $p \in (0,1)$. Since $\ell \leqslant \left\lfloor\frac{m-1}{2}\right\rfloor \implies 2\ell-m < 0$, this amounts to showing that the derivative of $f_{m}(\ell)$ with respect to $p$, for each $\ell \in \left\{0,1,\ldots,\left\lfloor\frac{m-1}{2}\right\rfloor\right\}$, is strictly negative for each $p \in (0,1)$. To this end, we prove the following claim:
\begin{equation}\label{objective_derivative}
\frac{d}{dp}f_{m}^{\ell}(p) = -\frac{m-2\ell}{2}\sum_{i=0}^{\ell}{m-\ell \choose i}{\ell \choose i}p^{2i}(1-p)^{m-1-2i} \text{ for each } \ell \in \left\{0,1,\ldots,\left\lfloor\frac{m-1}{2}\right\rfloor\right\}.
\end{equation} 
To prove \eqref{objective_derivative}, we fix an $\ell \in \left\{0,1,\ldots,\left\lfloor\frac{m-1}{2}\right\rfloor\right\}$ for the rest of the proof.

The idea employed in proving \eqref{objective_derivative} is the same as that inspiring Russo's formula (see, for instance, Section 2.4 of \cite{grimmett1999percolation}). Fix a vertex $v$ of $\mathbb{T}_{m}$, and let $v_{1}, \ldots, v_{m}$ denote its children. In order to take into account how the rate of change of $f_{m}^{\ell}(p)$ with respect to $p$ is impacted by the outcome of the experiment performed by the agent at each of $v_{1}, \ldots, v_{m}$ at time-step $t$, let us fix $p_{1}, p_{2}, \ldots, p_{m}$, each in $(0,1)$, and let $X_{t}(v_{1})$, $\ldots$, $X_{t}(v_{m})$ be mutually independent, but with $X_{t}(v_{i})$ now following Bernoulli$(p_{i})$ for each $i \in \{1,2,\ldots,m\}$. We let $\overline{p}=(p_{1},p_{2},\ldots,p_{m})$, and introduce the generalized definition 
\begin{equation}
f_{m}^{\ell}(\overline{p}) = \Prob\left[C_{t+1}(v)=B\big|C_{t}(v_{i})=B \text{ for each } i \in \{1,\ldots,\ell\} \text{ and } C_{t}(v_{i})=R \text{ for each } i \in \{\ell+1,\ldots,m\}\right],\nonumber
\end{equation}
where the rule for deciding $C_{t+1}(v)$, given $C_{t}(v_{1}), \ldots, C_{t}(v_{m})$, is the same as that described in \eqref{abs_maj_rule}. Note that, in our original set-up, we have $p_{1}=p_{2}=\cdots=p_{m}=p$, and $f_{m}^{\ell}(\overline{p})$ simply boils down to $f_{m}^{\ell}(p)$.

If we fix any $0 \leqslant \Delta(p_{i}) \leqslant 1-p_{i}$ for each $i \in \{1,2,\ldots,m\}$, and set $\Delta(\overline{p})=\left(\Delta(p_{1}),\Delta(p_{2}),\ldots,\Delta(p_{m})\right)$, then, by the well-known formula for total derivatives, we obtain
\begin{equation}
f_{m}^{\ell}\left(\overline{p}+\Delta(\overline{p})\right) - f_{m}^{\ell}(\overline{p}) = \sum_{i=1}^{m}\frac{\partial}{\partial p_{i}}f_{m}^{\ell}(\overline{p}) \cdot \Delta(p_{i}).\label{total_derivative}
\end{equation}
Note that, in our original set-up, where $p_{1}=p_{2}=\cdots=p_{m}=p$, we have, due to symmetry: 
\begin{equation}
\frac{\partial}{\partial p_{1}}f_{m}^{\ell}(\overline{p})\Big|_{(p,p,\ldots,p)} = \cdots = \frac{\partial}{\partial p_{\ell}}f_{m}^{\ell}(\overline{p})\Big|_{(p,p,\ldots,p)} \text{ and } \frac{\partial}{\partial p_{\ell+1}}f_{m}^{\ell}(\overline{p})\Big|_{(p,p,\ldots,p)} = \cdots = \frac{\partial}{\partial p_{m}}f_{m}^{\ell}(\overline{p})\Big|_{(p,p,\ldots,p)},\label{all_p_{i}_equal}
\end{equation}
which is why it suffices for us to find only $\frac{\partial}{\partial p_{1}}f_{m}^{\ell}(\overline{p})$ and $\frac{\partial}{\partial p_{\ell+1}}f_{m}^{\ell}(\overline{p})$. Furthermore, when $p_{1}=p_{2}=\cdots=p_{m}=p$ and $\Delta(p_{1})=\Delta(p_{2})=\cdots=\Delta(p_{m})=\Delta(p)$, we have 
\begin{equation}
\frac{d}{dp}f_{m}^{\ell}(p) = \lim_{\Delta(p) \rightarrow 0}\frac{1}{\Delta(p)}\left\{f_{m}^{\ell}\left(\overline{p}+\Delta(\overline{p})\right)\Big|_{(p+\Delta(p),p+\Delta(p),\ldots,p+\Delta(p))} - f_{m}^{\ell}(\overline{p})\Big|_{(p,p,\ldots,p)}\right\}.\label{derivative_as_limit}
\end{equation}

For each $i \in \{1,2,\ldots,m\}$, we set $\Delta^{(i)}(\overline{p})$ to be the $m$-tuple in which the $i$-th coordinate equals $\Delta(p_{i})$, and every other coordinate equals $0$. Note that computing $\frac{\partial}{\partial p_{i}}f_{m}^{\ell}(\overline{p})$ amounts to finding $f_{m}^{\ell}\left(\overline{p}+\Delta^{(i)}(\overline{p})\right)-f_{m}^{\ell}(\overline{p})$ for each $i \in \{1,2,\ldots,m\}$.

Let $U_{1}, U_{2}, \ldots, U_{m}$ denote i.i.d.\ Uniform$[0,1]$ random variables. When $X_{t}(v_{i})$ follows Bernoulli$(p_{i})$, we can write $X_{t}(v_{i}) = \mathbf{1}_{U_{i} \leqslant p_{i}}$, for each $i \in \{1,2,\ldots,m\}$. Consequently, analogous to \eqref{intermediate_1}, we obtain
\begin{align}
f_{m}^{\ell}(\overline{p}) ={}& \Prob\left[\sum_{i=1}^{\ell}\mathbf{1}_{U_{i} \leqslant p_{i}} > \sum_{i=\ell+1}^{m}\mathbf{1}_{U_{i} \leqslant p_{i}}\right] + \frac{1}{2}\Prob\left[\sum_{i=1}^{\ell}\mathbf{1}_{U_{i} \leqslant p_{i}} = \sum_{i=\ell+1}^{m}\mathbf{1}_{U_{i} \leqslant p_{i}}\right].\nonumber
\end{align}
We utilize this to now compute $f_{m}^{\ell}\left(\overline{p}+\Delta^{(1)}(\overline{p})\right)-f_{m}^{\ell}(\overline{p})$, considering various possible scenarios:
\begin{enumerate}
\item When $\sum_{i=2}^{\ell}\mathbf{1}_{U_{i} \leqslant p_{i}} > \sum_{i=\ell+1}^{m}\mathbf{1}_{U_{i}\leqslant p_{i}}$, since each of $\mathbf{1}_{U_{1} \leqslant p_{1}}$ and $\mathbf{1}_{U_{1} \leqslant p_{1}+\Delta(p_{1})}$ is non-negative, we have $\sum_{i=1}^{\ell}\mathbf{1}_{U_{i} \leqslant p_{i}} > \sum_{i=\ell+1}^{m}\mathbf{1}_{U_{i}\leqslant p_{i}}$ as well as $\mathbf{1}_{U_{1} \leqslant p_{1}+\Delta(p_{1})}+\sum_{i=2}^{\ell}\mathbf{1}_{U_{i} \leqslant p_{i}} > \sum_{i=\ell+1}^{m}\mathbf{1}_{U_{i}\leqslant p_{i}}$, which means that $\Prob\left[\sum_{i=2}^{\ell}\mathbf{1}_{U_{i} \leqslant p_{i}} > \sum_{i=\ell+1}^{m}\mathbf{1}_{U_{i}\leqslant p_{i}}\right]$ is present in each of $f_{m}^{\ell}\left(\overline{p}+\Delta^{(1)}(\overline{p})\right)$ and $f_{m}^{\ell}(\overline{p})$. Consequently, it leaves no contribution in $f_{m}^{\ell}\left(\overline{p}+\Delta^{(1)}(\overline{p})\right)-f_{m}^{\ell}(\overline{p})$.

\item When $\sum_{i=2}^{\ell}\mathbf{1}_{U_{i} \leqslant p_{i}} = \sum_{i=\ell+1}^{m}\mathbf{1}_{U_{i}\leqslant p_{i}}$, its contribution to $f_{m}^{\ell}(\overline{p})$ is given by 
\begin{align}
\Prob\left[U_{1} \leqslant p_{1},\sum_{i=2}^{\ell}\mathbf{1}_{U_{i} \leqslant p_{i}} = \sum_{i=\ell+1}^{m}\mathbf{1}_{U_{i}\leqslant p_{i}}\right]+\frac{1}{2}\Prob\left[U_{1} > p_{1},\sum_{i=2}^{\ell}\mathbf{1}_{U_{i} \leqslant p_{i}} = \sum_{i=\ell+1}^{m}\mathbf{1}_{U_{i}\leqslant p_{i}}\right],\nonumber
\end{align}
whereas its contribution to $f_{m}^{\ell}\left(\overline{p}+\Delta^{(1)}(\overline{p})\right)$ is given by 
\begin{align}
\Prob\left[U_{1} \leqslant p_{1}+\Delta(p_{1}),\sum_{i=2}^{\ell}\mathbf{1}_{U_{i} \leqslant p_{i}} = \sum_{i=\ell+1}^{m}\mathbf{1}_{U_{i}\leqslant p_{i}}\right]+\frac{1}{2}\Prob\left[U_{1} > p_{1}+\Delta(p_{1}),\sum_{i=2}^{\ell}\mathbf{1}_{U_{i} \leqslant p_{i}} = \sum_{i=\ell+1}^{m}\mathbf{1}_{U_{i}\leqslant p_{i}}\right].\nonumber
\end{align}
Combining these, the contribution of this scenario to $f_{m}^{\ell}\left(\overline{p}+\Delta^{(1)}(\overline{p})\right)-f_{m}^{\ell}(\overline{p})$ becomes 
\begin{equation}
\frac{1}{2}\Prob\left[p_{1} < U_{1} \leqslant p_{1}+\Delta(p_{1}),\sum_{i=2}^{\ell}\mathbf{1}_{U_{i} \leqslant p_{i}} = \sum_{i=\ell+1}^{m}\mathbf{1}_{U_{i}\leqslant p_{i}}\right] = \frac{\Delta(p_{1})}{2}\Prob\left[\sum_{i=2}^{\ell}\mathbf{1}_{U_{i} \leqslant p_{i}} = \sum_{i=\ell+1}^{m}\mathbf{1}_{U_{i}\leqslant p_{i}}\right].\nonumber
\end{equation}
\item When $\sum_{i=2}^{\ell}\mathbf{1}_{U_{i} \leqslant p_{i}} = \sum_{i=\ell+1}^{m}\mathbf{1}_{U_{i}\leqslant p_{i}}-1$, its contribution to $f_{m}^{\ell}(\overline{p})$ is given by
\begin{align}
\frac{1}{2}\Prob\left[U_{1} \leqslant p_{1}, \sum_{i=2}^{\ell}\mathbf{1}_{U_{i} \leqslant p_{i}} = \sum_{i=\ell+1}^{m}\mathbf{1}_{U_{i}\leqslant p_{i}}-1\right],\nonumber
\end{align}
whereas its contribution to $f_{m}^{\ell}\left(\overline{p}+\Delta^{(1)}(\overline{p})\right)$ is given by 
\begin{align}
\frac{1}{2}\Prob\left[U_{1} \leqslant p_{1}+\Delta(p_{1}), \sum_{i=2}^{\ell}\mathbf{1}_{U_{i} \leqslant p_{i}} = \sum_{i=\ell+1}^{m}\mathbf{1}_{U_{i}\leqslant p_{i}}-1\right].\nonumber
\end{align}
Thus, the contribution of this scenario to $f_{m}^{\ell}\left(\overline{p}+\Delta^{(1)}(\overline{p})\right)-f_{m}^{\ell}(\overline{p})$ equals
\begin{align}
\frac{1}{2}\Prob\left[p_{1} < U_{1} \leqslant p_{1}+\Delta(p_{1}),\sum_{i=2}^{\ell}\mathbf{1}_{U_{i} \leqslant p_{i}} = \sum_{i=\ell+1}^{m}\mathbf{1}_{U_{i}\leqslant p_{i}}-1\right] = \frac{\Delta(p_{1})}{2}\Prob\left[\sum_{i=2}^{\ell}\mathbf{1}_{U_{i} \leqslant p_{i}} = \sum_{i=\ell+1}^{m}\mathbf{1}_{U_{i}\leqslant p_{i}}-1\right].\nonumber
\end{align}
\item Finally, when $\sum_{i=2}^{\ell}\mathbf{1}_{U_{i} \leqslant p_{i}} < \sum_{i=\ell+1}^{m}\mathbf{1}_{U_{i}\leqslant p_{i}}-1$, its contribution to each of $f_{m}^{\ell}(\overline{p})$ and $f_{m}^{\ell}\left(\overline{p}+\Delta^{(1)}(\overline{p})\right)$ is $0$, so that its contribution to their difference is also $0$.
\end{enumerate}
Combining the contributions from the four scenarios considered above, we obtain:
\begin{align}
f_{m}^{\ell}\left(\overline{p}+\Delta^{(1)}(\overline{p})\right)-f_{m}^{\ell}(\overline{p}) ={}& \frac{\Delta(p_{1})}{2}\left\{\Prob\left[\sum_{i=2}^{\ell}\mathbf{1}_{U_{i} \leqslant p_{i}} = \sum_{i=\ell+1}^{m}\mathbf{1}_{U_{i}\leqslant p_{i}}\right]+\Prob\left[\sum_{i=2}^{\ell}\mathbf{1}_{U_{i} \leqslant p_{i}} = \sum_{i=\ell+1}^{m}\mathbf{1}_{U_{i}\leqslant p_{i}}-1\right]\right\}.\label{contribution_first_coordinate}
\end{align}

We now compute $f_{m}^{\ell}\left(\overline{p}+\Delta^{(\ell+1)}(\overline{p})\right)-f_{m}^{\ell}(\overline{p})$ via a case-by-case analysis in the same manner as above:
\begin{enumerate}
\item When $\sum_{i=1}^{\ell}\mathbf{1}_{U_{i}\leqslant p_{i}} > \sum_{i=\ell+2}^{m}\mathbf{1}_{U_{i}\leqslant p_{i}}+1$, then we have $\sum_{i=1}^{\ell}\mathbf{1}_{U_{i}\leqslant p_{i}} > \sum_{i=\ell+1}^{m}\mathbf{1}_{U_{i}\leqslant p_{i}}$ as well as $\sum_{i=1}^{\ell}\mathbf{1}_{U_{i}\leqslant p_{i}} > \mathbf{1}_{U_{\ell+1}\leqslant p_{\ell+1}}+\sum_{i=\ell+2}^{m}\mathbf{1}_{U_{i}\leqslant p_{i}}$. Thus, the contribution of this scenario to each of $f_{m}^{\ell}(\overline{p})$ and $f_{m}^{\ell}\left(\overline{p}+\Delta^{(\ell+1)}(\overline{p})\right)$ equals $\Prob\left[\sum_{i=1}^{\ell}\mathbf{1}_{U_{i}\leqslant p_{i}} > \sum_{i=\ell+2}^{m}\mathbf{1}_{U_{i}\leqslant p_{i}}+1\right]$, and therefore, its contribution to their difference is $0$.
\item When $\sum_{i=1}^{\ell}\mathbf{1}_{U_{i}\leqslant p_{i}} = \sum_{i=\ell+2}^{m}\mathbf{1}_{U_{i}\leqslant p_{i}}+1$, its contribution to $f_{m}^{\ell}(\overline{p})$ is 
\begin{align}
\Prob\left[\sum_{i=1}^{\ell}\mathbf{1}_{U_{i}\leqslant p_{i}} = \sum_{i=\ell+2}^{m}\mathbf{1}_{U_{i}\leqslant p_{i}}+1,U_{\ell+1}>p_{\ell+1}\right]+\frac{1}{2}\Prob\left[\sum_{i=1}^{\ell}\mathbf{1}_{U_{i}\leqslant p_{i}} = \sum_{i=\ell+2}^{m}\mathbf{1}_{U_{i}\leqslant p_{i}}+1,U_{\ell+1}\leqslant p_{\ell+1}\right].\nonumber
\end{align}
Likewise, the contribution of this scenario to $f_{m}^{\ell}\left(\overline{p}+\Delta^{(\ell+1)}(\overline{p})\right)$ equals
\begin{align}
{}&\Prob\left[\sum_{i=1}^{\ell}\mathbf{1}_{U_{i}\leqslant p_{i}} = \sum_{i=\ell+2}^{m}\mathbf{1}_{U_{i}\leqslant p_{i}}+1,U_{\ell+1}>p_{\ell+1}+\Delta(p_{\ell+1})\right]\nonumber\\&+\frac{1}{2}\Prob\left[\sum_{i=1}^{\ell}\mathbf{1}_{U_{i}\leqslant p_{i}} = \sum_{i=\ell+2}^{m}\mathbf{1}_{U_{i}\leqslant p_{i}}+1,U_{\ell+1}\leqslant p_{\ell+1}+\Delta(p_{\ell+1})\right].\nonumber
\end{align}
Consequently, the contribution of this scenario to $f_{m}^{\ell}\left(\overline{p}+\Delta^{(\ell+1)}(\overline{p})\right)-f_{m}^{\ell}(\overline{p})$ equals
\begin{align}
{}&-\frac{1}{2}\Prob\left[\sum_{i=1}^{\ell}\mathbf{1}_{U_{i}\leqslant p_{i}} = \sum_{i=\ell+2}^{m}\mathbf{1}_{U_{i}\leqslant p_{i}}+1,p_{\ell+1}< U_{\ell+1} \leqslant p_{\ell+1}+\Delta(p_{\ell+1})\right]\nonumber\\&=-\frac{\Delta(p_{\ell+1})}{2}\Prob\left[\sum_{i=1}^{\ell}\mathbf{1}_{U_{i}\leqslant p_{i}} = \sum_{i=\ell+2}^{m}\mathbf{1}_{U_{i}\leqslant p_{i}}+1\right].\nonumber
\end{align}
\item When $\sum_{i=1}^{\ell}\mathbf{1}_{U_{i}\leqslant p_{i}} = \sum_{i=\ell+2}^{m}\mathbf{1}_{U_{i}\leqslant p_{i}}$, its contribution to $f_{m}^{\ell}(\overline{p})$ is 
\begin{align}
\frac{1}{2}\Prob\left[\sum_{i=1}^{\ell}\mathbf{1}_{U_{i}\leqslant p_{i}} = \sum_{i=\ell+2}^{m}\mathbf{1}_{U_{i}\leqslant p_{i}}, U_{\ell+1} > p_{\ell+1}\right],\nonumber
\end{align}
and its contribution to $f_{m}^{\ell}\left(\overline{p}+\Delta^{(\ell+1)}(\overline{p})\right)$ equals
\begin{align}
\frac{1}{2}\Prob\left[\sum_{i=1}^{\ell}\mathbf{1}_{U_{i}\leqslant p_{i}} = \sum_{i=\ell+2}^{m}\mathbf{1}_{U_{i}\leqslant p_{i}}, U_{\ell+1} > p_{\ell+1}+\Delta(p_{\ell+1})\right].\nonumber
\end{align}
The contribution of this scenario to $f_{m}^{\ell}\left(\overline{p}+\Delta^{(\ell+1)}(\overline{p})\right)-f_{m}^{\ell}(\overline{p})$ is thus
\begin{align}
{}&-\frac{1}{2}\Prob\left[p_{\ell+1}<U_{\ell+1}\leqslant p_{\ell+1}+\Delta(p_{\ell+1}),\sum_{i=1}^{\ell}\mathbf{1}_{U_{i}\leqslant p_{i}} = \sum_{i=\ell+2}^{m}\mathbf{1}_{U_{i}\leqslant p_{i}}\right]=-\frac{\Delta(p_{\ell+1})}{2}\Prob\left[\sum_{i=1}^{\ell}\mathbf{1}_{U_{i}\leqslant p_{i}} = \sum_{i=\ell+2}^{m}\mathbf{1}_{U_{i}\leqslant p_{i}}\right].\nonumber
\end{align}
\item Finally, when $\sum_{i=1}^{\ell}\mathbf{1}_{U_{i}\leqslant p_{i}} < \sum_{i=\ell+2}^{m}\mathbf{1}_{U_{i}\leqslant p_{i}}$, its contribution to each of $f_{m}^{\ell}\left(\overline{p}+\Delta^{(\ell+1)}(\overline{p})\right)$ and $f_{m}^{\ell}(\overline{p})$ is $0$, so that its contribution to their difference also equals $0$.
\end{enumerate}
Combining the contributions from the four scenarios considered above, we obtain:
\begin{align}
f_{m}^{\ell}\left(\overline{p}+\Delta^{(\ell+1)}(\overline{p})\right)-f_{m}^{\ell}(\overline{p})={}&-\frac{\Delta(p_{\ell+1})}{2}\left\{\Prob\left[\sum_{i=1}^{\ell}\mathbf{1}_{U_{i}\leqslant p_{i}} = \sum_{i=\ell+2}^{m}\mathbf{1}_{U_{i}\leqslant p_{i}}+1\right]+\Prob\left[\sum_{i=1}^{\ell}\mathbf{1}_{U_{i}\leqslant p_{i}} = \sum_{i=\ell+2}^{m}\mathbf{1}_{U_{i}\leqslant p_{i}}\right]\right\}.\label{contribution_ell+1_coordinate}
\end{align}

We come back to our set-up, i.e.\ where $p_{1}=p_{2}=\cdots=p_{m}=p$ and $\Delta(p_{1})=\Delta(p_{2})=\cdots=\Delta(p_{m})=\Delta(p)$. Setting $Z=\sum_{i=2}^{\ell}\mathbf{1}_{U_{i}\leqslant p}$ and $W=\sum_{i=\ell+2}^{m}\mathbf{1}_{U_{i}\leqslant p}$ for the sake of brevity, and recalling that $X_{t}(v_{1})=\mathbf{1}_{U_{1}\leqslant p}$ and $X_{t}(v_{\ell+1})=\mathbf{1}_{U_{\ell+1} \leqslant p}$, we have, from \eqref{total_derivative}, \eqref{all_p_{i}_equal}, \eqref{contribution_first_coordinate} and \eqref{contribution_ell+1_coordinate}:
\begin{align}
{}& f_{m}^{\ell}(\overline{p}+\Delta(\overline{p}))\Big|_{(p+\Delta(p),p+\Delta(p),\ldots,p+\Delta(p))}-f_{m}^{\ell}(\overline{p})\Big|_{(p,p,\ldots,p)}\nonumber\\
={}& \frac{\ell \Delta(p)}{2}\Prob\left[Z=X_{t}(v_{\ell+1})+W\right]+\frac{\ell \Delta(p)}{2}\Prob\left[Z=X_{t}(v_{\ell+1})+W-1\right]\nonumber\\&-\frac{(m-\ell) \Delta(p)}{2}\Prob\left[X_{t}(v_{1})+Z=W+1\right]-\frac{(m-\ell) \Delta(p)}{2}\Prob\left[X_{t}(v_{1})+Z=W\right]\nonumber\\
={}&\frac{\Delta(p)}{2}\Big\{\frac{\ell}{1-p}\Prob\left[Z=X_{t}(v_{\ell+1})+W,X_{t}(v_{1})=0\right]+\ell\Prob\left[Z=X_{t}(v_{\ell+1})+W-1\right]\nonumber\\&-(m-\ell)\Prob\left[X_{t}(v_{1})+Z=W+1\right]-\frac{m-\ell}{1-p}\Prob\left[X_{t}(v_{1})+Z=W,X_{t}(v_{\ell+1})=0\right]\Big\}\nonumber\\
={}&\frac{\Delta(p)}{2}\Big\{\frac{\ell}{1-p}\Prob\left[X_{t}(v_{1})+Z=X_{t}(v_{\ell+1})+W,X_{t}(v_{1})=0\right]+\ell\Prob\left[Z=X_{t}(v_{\ell+1})+W-1\right]\nonumber\\&-(m-\ell)\Prob\left[X_{t}(v_{1})+Z=W+1\right]-\frac{m-\ell}{1-p}\Prob\left[X_{t}(v_{1})+Z=W+X_{t}(v_{\ell+1}),X_{t}(v_{\ell+1})=0\right]\Big\}\nonumber\\
={}&\frac{\Delta(p)}{2}\Big\{\frac{\ell}{1-p}\Prob\left[X_{t}(v_{1})+Z=X_{t}(v_{\ell+1})+W\right]-\frac{\ell}{1-p}\Prob\left[X_{t}(v_{1})+Z=X_{t}(v_{\ell+1})+W,X_{t}(v_{1})=1\right]\nonumber\\&+\ell\Prob\left[1+Z=X_{t}(v_{\ell+1})+W\right]-(m-\ell)\Prob\left[X_{t}(v_{1})+Z=W+1\right]-\frac{m-\ell}{1-p}\Prob\left[X_{t}(v_{1})+Z=W+X_{t}(v_{\ell+1})\right]\nonumber\\&+\frac{m-\ell}{1-p}\Prob\left[X_{t}(v_{1})+Z=W+X_{t}(v_{\ell+1}),X_{t}(v_{\ell+1})=1\right]\Big\}\nonumber\\
={}&\frac{\Delta(p)}{2}\Big\{-\frac{(m-2\ell)}{1-p}\Prob\left[X_{t}(v_{1})+Z=X_{t}(v_{\ell+1})+W\right]-\frac{\ell p}{1-p}\Prob\left[1+Z=X_{t}(v_{\ell+1})+W\right]\nonumber\\&+\ell\Prob\left[1+Z=X_{t}(v_{\ell+1})+W\right]-(m-\ell)\Prob\left[X_{t}(v_{1})+Z=W+1\right]+\frac{(m-\ell)p}{1-p}\Prob\left[X_{t}(v_{1})+Z=W+1\right]\Big\}\nonumber\\
={}&\frac{\Delta(p)}{2}\Big\{-\frac{(m-2\ell)}{1-p}\Prob\left[X_{t}(v_{1})+Z=X_{t}(v_{\ell+1})+W\right]+\frac{\ell(1-2p)}{1-p}\Prob\left[1+Z=X_{t}(v_{\ell+1})+W\right]\nonumber\\&-\frac{(m-\ell)(1-2p)}{1-p}\Prob\left[X_{t}(v_{1})+Z=W+1\right]\Big\}.\label{intermediate_8}
\end{align}
In the computation above, we have made use of the mutual independence of the random variables $X_{t}(v_{1})$, $Z$, $X_{t}(v_{\ell+1})$ and $W$. Note that, in our set-up (i.e.\ where $p_{1}=p_{2}=\cdots=p_{m}=p$ and $X_{t}(v_{i})$ are i.i.d.\ Bernoulli$(p)$ for all $i \in \{1,2,\ldots,m\}$), $Z$ follows Binomial$(\ell-1,p)$, $W$ follows Binomial$(m-\ell-1,p)$, $X_{t}(v_{\ell+1})+W$ follows Binomial$(m-\ell,p)$, and $X_{t}(v_{1})+Z$ follows Binomial$(\ell,p)$). Moreover, $Z$ and $X_{t}(v_{\ell+1})+W$ are independent of each other, as are $X_{t}(v_{1})+Z$ and $W$. We now focus on the second part of the final expression of \eqref{intermediate_8}, recalling from \eqref{objective_derivative} that we consider $\ell \leqslant \left\lfloor\frac{m-1}{2}\right\rfloor$, which implies $\ell < m-\ell$:
\begin{align}
{}&\ell\Prob\left[1+Z=X_{t}(v_{\ell+1})+W\right]-(m-\ell)\Prob\left[X_{t}(v_{1})+Z=W+1\right]\nonumber\\
={}&\ell\sum_{r=0}^{\ell-1}\Prob\left[Z=r,X_{t}(v_{\ell+1})+W=r+1\right]-(m-\ell)\sum_{r=0}^{\ell-1}\Prob\left[X_{t}(v_{1})+Z=r+1,W=r\right]\nonumber\\
={}&\ell\sum_{r=0}^{\ell-1}{\ell-1 \choose r}p^{r}(1-p)^{\ell-1-r}{m-\ell \choose r+1}p^{r+1}(1-p)^{m-\ell-r-1}\nonumber\\&-(m-\ell)\sum_{r=0}^{\ell-1}{\ell \choose r+1}p^{r+1}(1-p)^{\ell-r-1}{m-\ell-1 \choose r}p^{r}(1-p)^{m-\ell-1-r} = 0.\nonumber
\end{align}
Incorporating this observation into \eqref{intermediate_8}, and applying \eqref{derivative_as_limit}, we obtain:
\begin{align}
\frac{d}{dp}f_{m}^{\ell}(p) ={}& -\frac{(m-2\ell)}{2(1-p)}\Prob\left[X_{t}(v_{1})+Z=X_{t}(v_{\ell+1})+W\right]=-\frac{(m-2\ell)}{2(1-p)}\sum_{r=0}^{\ell}\Prob\left[X_{t}(v_{1})+Z=X_{t}(v_{\ell+1})+W=r\right]\nonumber\\
={}&-\frac{m-2\ell}{2}\sum_{r=0}^{\ell}{\ell \choose r}{m-\ell \choose r}p^{2r}(1-p)^{m-1-2r},\nonumber
\end{align}
which is precisely the identity in \eqref{objective_derivative} that we set out to establish. From \eqref{g'_{m}(1/2)} and \eqref{objective_derivative}, we conclude that $\frac{d}{dp}g'_{m}(1/2)\big|_{p} > 0$ for each $p \in (0,1)$. This completes the proof of Theorem~\ref{thm:g'_{m}(1/2)_increasing}.
\end{proof}

As explained in \eqref{broad_step_4}, our task now is to show that $g'_{m}(1/2)\big|_{p=0} < 1$ and $g'_{m}(1/2)\big|_{p=1} > 1$. Recall, from the very last paragraph of the proof of Theorem~\ref{thm:convex_concave}, that $f_{m}^{\ell}(1)=0$ for each $\ell \leqslant \left\lfloor\frac{m-1}{2}\right\rfloor$. This, along with \eqref{g'_{m}(1/2)}, yields 
\begin{equation}
g'_{m}\left(\frac{1}{2}\right)\Big|_{p=1}=\frac{m}{2^{m-2}}{m-1 \choose \left\lfloor\frac{m-1}{2}\right\rfloor}\nonumber
\end{equation}
Comparing the values of $g'_{m}(1/2)\big|_{p=1}$ for $m=2n+1$ and $m=2(n+1)+1$, we obtain
\begin{equation}
\frac{g'_{2(n+1)+1}\left(\frac{1}{2}\right)\big|_{p=1}}{g'_{2n+1}\left(\frac{1}{2}\right)\big|_{p=1}}=\frac{2n+3}{2^{2n+1}}{2n+2 \choose n+1} \left(\frac{2n+1}{2^{2n-1}}{2n \choose n}\right)^{-1} = \frac{2n+3}{2n+2} > 1\nonumber
\end{equation}
for each $n \in \mathbb{N}_{0}$, so that $\left\{g'_{2n+1}(1/2)\big|_{p=1}\right\}_{n \in \mathbb{N}}$ is a strictly increasing sequence. Likewise, comparing the values of $g'_{m}(1/2)\big|_{p=1}$ for $m=2n$ and $m=2(n+1)$, we obtain
\begin{equation}
\frac{g'_{2(n+1)}\left(\frac{1}{2}\right)\big|_{p=1}}{g'_{2n}\left(\frac{1}{2}\right)\big|_{p=1}}=\frac{2n+2}{2^{2n}}{2n+1 \choose n}\left(\frac{2n}{2^{2n-2}}{2n-1 \choose n-1}\right)^{-1} = \frac{2n+1}{2n} > 1,\nonumber
\end{equation}
for each $n \in \mathbb{N}$, so that $\left\{g'_{2n}(1/2)\big|_{p=1}\right\}_{n \in \mathbb{N}}$ is a strictly increasing sequence. Moreover, we have $g'_{3}(1/2)\big|_{p=1}=3$ and $g'_{2}(1/2)\big|_{p=1}=2$. These observations, together, allow us to conclude that $g'_{m}(1/2)\big|_{p=1} > 1$ for every $m \in \mathbb{N}$ with $m \geqslant 2$. Finally, recall, from Remark~\ref{rem:p=0}, that when $p=0$, the function $g_{m}(x)=1/2$ for all $x \in [0,1]$, so that $g'_{m}(1/2)=0$, for every $m \in \mathbb{N}$ with $m \geqslant 2$. As explained in \eqref{broad_step_5}, the proofs of \eqref{thm:equal_part_1} and \eqref{thm:equal_part_2} of Theorem~\ref{thm:main_equal} are now complete. 

We now come to the proof of the third and last assertion made in the statement of Theorem~\ref{thm:main_equal}. To this end, we first state and prove a lemma that will be useful in other regimes of values of $(p_{B},p_{R})$:
\begin{lemma}\label{lem:increasing_function_limits}
Let $h$ be a strictly increasing, continuous function on $[0,1]$, with $h(0) \geqslant 0$, $h(1) \leqslant 1$, and three distinct fixed points, namely $0 \leqslant \alpha_{1} < \alpha_{2} < \alpha_{3} \leqslant 1$. Fix any $\gamma_{0} \in [0,1]$, and let $\gamma_{n+1}=h(\gamma_{n})$ for each $n \in \mathbb{N}_{0}$. Then the limit $\gamma = \lim_{n \rightarrow \infty}\gamma_{n}$ exists and equals 
\begin{enumerate*}
\item $\alpha_{1}$ when $\gamma_{0} \in [0,\alpha_{2})$,
\item $\alpha_{3}$ when $\gamma_{0} \in (\alpha_{2},1]$, and
\item $\alpha_{2}$ when $\gamma_{0}=\alpha_{2}$.
\end{enumerate*}
\end{lemma}
\begin{proof}
We first observe that $0 \leqslant h(x) \leqslant 1$ for each $x \in [0,1]$. We note that the curve $y=h(x)$ lies 
\begin{enumerate}
\item above the line $y=x$ for $x \in [0,\alpha_{1})$ (if $\alpha_{1}=0$, which is equivalent to saying that $h(0)=0$, then this sub-interval simply does not exist),
\item beneath the line $y=x$ for $x \in (\alpha_{1},\alpha_{2})$,
\item above the line $y=x$ for $x \in (\alpha_{2},\alpha_{3})$, and
\item beneath the line $y=x$ for $x \in (\alpha_{3},1]$ (if $\alpha_{3}=1$, which is equivalent to saying that $h(1)=1$, then this sub-interval simply does not exist).
\end{enumerate}
Note, also, that whenever the limit $\gamma$, as defined in the statement of Lemma~\ref{lem:increasing_function_limits}, exists, it must be a fixed point of $h$. This is because 
\begin{equation}
h(\gamma) = h\left(\lim_{n \rightarrow \infty}\gamma_{n}\right) = \lim_{n \rightarrow \infty}h(\gamma_{n}) = \lim_{n \rightarrow \infty}\gamma_{n+1} = \gamma,\nonumber
\end{equation}
since $h$ is continuous.

Let us consider $\gamma_{0} \in [0,\alpha_{1})$ (when this sub-interval is non-empty, i.e.\ when $h(0) > 0$). The curve $y=h(x)$ lies above the line $y=x$ for $x \in [0,\alpha_{1})$, so that $h(\gamma_{0}) > \gamma_{0}$, and since $h$ is strictly increasing, an iterative application of $h$ yields $\gamma_{n+1} > \gamma_{n}$ for each $n \in \mathbb{N}_{0}$. Thus, the sequence $\{\gamma_{n}\}_{n \in \mathbb{N}_{0}}$ is strictly increasing. It is also bounded between $0$ and $1$, which implies that the limit $\gamma = \lim_{n \rightarrow \infty}\gamma_{n}$ exists. Since $\gamma_{0} < \alpha_{1}$, the strictly increasing nature of $h$ also ensures that $\gamma_{n} < \alpha_{1}$ for each $n \in \mathbb{N}_{0}$. This, in turn, implies that $\gamma \leqslant \alpha_{1}$ as well. Since $\gamma$ must be a fixed point of $h$, the only possibility is that $\gamma=\alpha_{1}$.

When $\gamma_{0} \in (\alpha_{2},\alpha_{3})$, we argue in the same manner as above that $\gamma = \lim_{n \rightarrow \infty}\gamma_{n}$ exists, and that it equals $\alpha_{3}$.

Let us now consider $\gamma_{0} \in (\alpha_{1},\alpha_{2})$. The curve $y=h(x)$ lies beneath the line $y=x$ for $x \in (\alpha_{1},\alpha_{2})$, so that $h(\gamma_{0}) < \gamma_{0}$, and since $h$ is strictly increasing, an iterative application of $h$ yields $\gamma_{n+1} < \gamma_{n}$ for each $n \in \mathbb{N}_{0}$. Thus, the sequence $\{\gamma_{n}\}_{n \in \mathbb{N}_{0}}$ is strictly decreasing. It is also bounded between $0$ and $1$, which implies that the limit $\gamma = \lim_{n \rightarrow \infty}\gamma_{n}$ exists. Since $\gamma_{n} < \gamma_{0}$, and $\alpha_{1} < \gamma_{0}$ implies, via an interative application of the strictly increasing $h$, that $\alpha_{1} < \gamma_{n}$, we conclude that $\alpha_{1} < \gamma_{n} < \gamma_{0}$ for each $n$, so that $\alpha_{1} \leqslant \gamma < \gamma_{0}$. Since the only fixed point in $[\alpha_{1},\gamma_{0})$ is $\alpha_{1}$, hence $\gamma=\alpha_{1}$ in this case.

When $\gamma_{0} \in (\alpha_{3},1]$ (when this sub-interval is non-empty, i.e.\ when $h(1) < 1$), we argue in the same manner as in the previous paragraph that $\gamma = \lim_{n \rightarrow \infty}\gamma_{n}$ exists, and that it equals $\alpha_{3}$. 

Finally, if $\gamma_{0}=\alpha_{i}$ for any $i \in \{1,2,3\}$, then $\gamma_{n}=\alpha_{i}$ for each $n \in \mathbb{N}$, so that $\gamma = \alpha_{i}$ as well. This completes the proof of Lemma~\ref{lem:increasing_function_limits}.
\end{proof}

As evident from Remark~\ref{rem:p_{B}=p_{R}_unique_regime_already_proved} and \eqref{thm:equal_part_1} of Theorem~\ref{thm:main_equal}, for every choice of $\pi_{0}$ in \eqref{initial}, the sequence $\{\pi_{t}\}_{t \in \mathbb{N}_{0}}$, defined via \eqref{time_t}, converges to the unique fixed point $1/2$ of $g_{m}$ as $t \rightarrow \infty$. Therefore, we need only focus on the regime $p(m) < p < 1$, where the fixed points of $g_{m}$, in the interval $[0,1]$, are of the form $\alpha$, $1/2$ and $1-\alpha$, for some $\alpha \in [0,1/2)$, by part \eqref{thm:equal_part_2} of Theorem~\ref{thm:main_equal}. We begin by showing that, for a fixed $m \in \mathbb{N}$ with $m \geqslant 2$ and a fixed $p \in (0,1]$, the function $g_{m}$ is strictly increasing on $[0,1]$ (when $p=0$, we have already seen from Remark~\ref{rem:p=0} that $g_{m}$ equals the constant function $1/2$). Note, from \eqref{g'(x)} and \eqref{intermediate_1}, that
\begin{align}
{}&g'_{m}(1) = m\left\{f_{m}(m)-f_{m}(m-1)\right\} = m\left\{\Prob[A_{m}>0]+\frac{1}{2}\Prob[A_{m}=0]-\Prob[A_{m-1}>B_{1}]-\frac{1}{2}\Prob[A_{m-1}=B_{1}]\right\}\nonumber\\
={}& m\left\{1-\frac{1}{2}\Prob[A_{m}=0]-p\Prob[A_{m-1}>1]-(1-p)\Prob[A_{m-1}>0]-\frac{p}{2}\Prob[A_{m-1}=1]-\frac{(1-p)}{2}\Prob[A_{m-1}=0]\right\}\nonumber\\
={}& m\left\{\frac{p}{2}\Prob[A_{m-1}=1]+\frac{1+p}{2}\Prob[A_{m-1}=0]-\frac{1}{2}\Prob[A_{m}=0]\right\} \nonumber\\
={}& \frac{m}{2}\left\{(m-1)p^{2}(1-p)^{m-2}+(1+p)(1-p)^{m-1}-(1-p)^{m}\right\} > 0.\nonumber
\end{align}
Moreover, applying \eqref{symm_cond} to the expression in \eqref{g'(x)} and subsequently applying a change of variable to the summation (switching the index of the sum from $\ell$ to $m-1-\ell$), we obtain $g'_{m}(1-x)=g'_{m}(x)$, so that $g'_{m}(0) > 0$. From Theorem~\ref{thm:convex_concave}, we know that $g''_{m}(x) > 0$ for each $x \in [0,1/2)$ and $g''_{m}(x) < 0$ for each $x \in (1/2,1]$, so that $g'_{m}(x)$ is strictly increasing for $x \in [0,1/2)$ and $g'_{m}(x)$ is strictly decreasing for $x \in (1/2,1]$. Since we have shown above that both $g'_{m}(1)$ and $g'_{m}(0)$ are strictly positive, it is evident that $g'_{m}(x) > 0$ for all $x \in [0,1]$, allowing us to conclude that $g_{m}(x)$ is strictly increasing for $x \in [0,1]$. The rest of the last assertion made in the statement of Theorem~\ref{thm:main_equal} is established simply by an application of Lemma~\ref{lem:increasing_function_limits}.

\subsection{Finding the threshold probability $p(m)$ for small values of $m$}\label{subsec:small_m}
For small values of $m$, it is relatively easy to find, analytically, the value of the threshold $p(m)$, in Theorem~\ref{thm:main_equal}. When $m=3$, we obtain, from \eqref{intermediate_1}, that 
\begin{equation}
f_{3}(0) = \frac{1}{2}\Prob[B_{3}=0] = \frac{1}{2}(1-p)^{3}\nonumber
\end{equation}
and
\begin{align}
f_{3}(1) ={}& \Prob[B_{2}=0,A_{1}=1]+\frac{1}{2}\Prob[B_{2}=A_{1}=1]+\frac{1}{2}\Prob[B_{2}=A_{1}=0]\nonumber\\
={}& p(1-p)^{2}+p^{2}(1-p)+\frac{1}{2}(1-p)^{3}.\nonumber
\end{align}
Since \eqref{symm_cond} holds, we need not compute $f_{3}(2)$ and $f_{3}(3)$. From \eqref{g_{m}^{abs}}, we have 
\begin{align}
{}& g_{3}(x) = f_{3}(0)(1-x)^{3} + 3f_{3}(1)x(1-x)^{2} + 3\left\{1-f_{3}(1)\right\}x^{2}(1-x) + \left\{1-f_{3}(0)\right\}x^{3}\nonumber\\
={}& \frac{1}{2}(1-p)^{3}(1-2x)(1-x+x^{2}) + 3\left\{p(1-p)^{2}+p^{2}(1-p)+\frac{1}{2}(1-p)^{3}\right\}x(1-x)(1-2x) + 3x^{2}-2x^{3},\nonumber
\end{align}
so that, upon differentation, we obtain
\begin{align}
g'_{3}(x) ={}& -\frac{3}{2}(1-p)^{3}(1-2x+2x^{2}) + \frac{3}{2}(1-p)(1+p^{2})(1-6x+6x^{2}) + 6x-6x^{2}.\nonumber
\end{align}
From Lemma~\ref{lem:fixed_point_count_g'(1/2)}, it suffices for us to find all values of $p$ for which the derivative $g'_{3}(1/2) \leqslant 1$, which is equivalent to
\begin{align}
{}& -\frac{3}{2}(1-p)^{3}\left\{1-2\left(\frac{1}{2}\right)+2\left(\frac{1}{2}\right)^{2}\right\} + \frac{3}{2}(1-p)(1+p^{2})\left\{1-6\left(\frac{1}{2}\right)+6\left(\frac{1}{2}\right)^{2}\right\} + 6\left(\frac{1}{2}\right)-6\left(\frac{1}{2}\right)^{2} \leqslant 1\nonumber\\
{}&\Longleftrightarrow \frac{3p^{3}}{2}-3p^{2}+3p \leqslant 1 \Longleftrightarrow p \leqslant \frac{1}{3}(2+2^{1/3}-2^{2/3}) \approx 0.557507.\nonumber
\end{align}
That the threshold $p(3)$ indeed equals $0.557507$ can be deduced analytically as follows: standard techniques for finding the roots of cubic polynomials yield the only real root of the polynomial $3/2p^{3}-3p^{2}+3p-1$ to be $1/3(2+2^{1/3}-2^{2/3}) \approx 0.557507$, and since this polynomial takes the value $-1$ at $p=0$, it is evident that it it negative for all $p < 0.557507$ and positive for all $p > 0.557507$.
Thus, when $m=3$ and $p_{B}=p_{R}=p$, the function $g_{3}$ has a unique fixed point, which is $1/2$, in $[0,1]$ if and only if $p \leqslant 0.557507$. For $p > 0.557507$, the function $g_{3}$ has three distinct fixed points in $[0,1]$, of the form $\alpha$, $1/2$ and $1-\alpha$ for some $\alpha \in [0,1/2)$.

When $m=4$, we obtain, from \eqref{intermediate_1}, 
\begin{equation}
f_{4}(0) = \frac{1}{2}\Prob[B_{4}=0] = \frac{1}{2}(1-p)^{4}\nonumber
\end{equation}
and
\begin{equation}
f_{4}(1) = \Prob[A_{1}=1,B_{3}=0] + \frac{1}{2}\Prob[A_{1}=B_{3}=1] + \frac{1}{2}\Prob[A_{1}=B_{3}=0] = p(1-p)^{3}+\frac{3}{2}p^{2}(1-p)^{2} + \frac{1}{2}(1-p)^{4},\nonumber
\end{equation}
so that \eqref{g_{m}^{abs}}, together with \eqref{symm_cond} (which also implies that $f_{4}(2) = 1/2$), yields
\begin{align}
{}&g_{4}(x) = f_{4}(0)(1-x)^{4} + 4 f_{4}(1)x(1-x)^{3} + 6 \cdot \frac{1}{2} x^{2}(1-x)^{2} + 4\left\{1-f_{4}(1)\right\}x^{3}(1-x) + \left\{1-f_{4}(0)\right\}x^{4}\nonumber\\
={}& \frac{1}{2}(1-p)^{4}(1-2x)(1-2x+2x^{2}) + 4\left\{p(1-p)^{3}+\frac{3}{2}p^{2}(1-p)^{2} + \frac{1}{2}(1-p)^{4}\right\}x(1-x)(1-2x) + 3x^{2}-2x^{3}.\nonumber
\end{align}
Differentiating, we obtain
\begin{align}
g'_{4}(x) ={}& -2(1-p)^{4}(1-3x+3x^{2}) + 2(1-p)^{2}(1+2p^{2})(1-6x+6x^{2}) + 6x-6x^{2},\nonumber
\end{align}
so that the inequality that $g'_{4}(1/2) \leqslant 1$ becomes equivalent to
\begin{align}
{}& -2(1-p)^{4}\left\{1-3\left(\frac{1}{2}\right)+3\left(\frac{1}{2}\right)^{2}\right\} + 2(1-p)^{2}(1+2p^{2})\left\{1-6\left(\frac{1}{2}\right)+6\left(\frac{1}{2}\right)^{2}\right\} + 6\left(\frac{1}{2}\right)-6\left(\frac{1}{2}\right)^{2} \leqslant 1\nonumber\\
{}&\Longleftrightarrow 4p-6p^{2}+6p^{3}-\frac{5p^{4}}{2} \leqslant 1 \Longleftrightarrow p \leqslant 0.42842.\nonumber
\end{align}
That the threshold $p(4)$ indeed equals $0.42842$ can be justified analytically as follows: standard techniques for solving for the roots of a quartic equation yield the only two real roots of the polynomial $4p-6p^{2}+6p^{3}-(5/2)p^{4}-1$ to be approximately $0.42842$ and $1.3300$. Since the polynomial $4p-6p^{2}+6p^{3}-(5/2)p^{4}-1$ assumes the value $-1$ at $p=0$, it is immediate that it is negative for all $p \in [0,0.42842)$ and positive for all $p \in (0.42842,1]$. Thus, when $m=4$ and $p_{B}=p_{R}=p$, the function $g_{4}$ has a unique fixed point, which is $1/2$, in $[0,1]$ if and only if $p \leqslant 0.42842$. For $p > 0.42842$, the function $g_{4}$ has three distinct fixed points in $[0,1]$, of the form $\alpha$, $1/2$ and $1-\alpha$ for some $\alpha \in [0,1/2)$.

\section{Analysis of our learning model when $m=3$, $p_{B}=1$ and $p_{R} \in [0,1]$}\label{sec:p_{B}=1}
We prove Theorem~\ref{thm:unequal} in \S\ref{sec:p_{B}=1}. We begin by computing $f_{3}(k)$ for $k \in \{0,1,2,3\}$. Since $p_{B}=1$, we have $A_{k}=k$ in \eqref{intermediate_1} for each $k \in \{0,1,2,3\}$, so that 
\begin{equation}
f_{3}(0) = \frac{1}{2}\Prob[B_{3}=0] = \frac{1}{2}(1-p_{R})^{3} \text{ and } f_{3}(1) = \Prob[B_{2}<1]+\frac{1}{2}\Prob[B_{2}=1] = (1-p_{R}),\nonumber
\end{equation}
whereas $f_{3}(2)=f_{3}(3)=1$. This yields, from \eqref{g_{m}^{abs}},
\begin{align}
g_{3}(x) ={}& \sum_{k=0}^{3}f_{3}(k){3 \choose k}x^{k}(1-x)^{3-k} = \frac{1}{2}(1-p_{R})^{3}(1-x)^{3}+3(1-p_{R})x(1-x)^{2}+3x^{2}(1-x)+x^{3}.\label{g_{3}(x)_p_{B}=1}
\end{align}

From \eqref{g_{3}(x)_p_{B}=1}, we obtain 
\begin{align}
g_{3}(x)-x ={}& \frac{1}{2}(1-p_{R})^{3}(1-x)^{3}+3(1-p_{R})x(1-x)^{2}+3x^{2}(1-x)+x^{3}-x\nonumber\\
={}& (1-x)\left\{\frac{1}{2}(1-p_{R})^{3}(1-x)^{2}+3(1-p_{R})x(1-x)+3x^{2}-x(1+x)\right\}\nonumber\\
={}& (1-x)\left[\left\{\frac{1}{2}(1-p_{R})^{3}-3(1-p_{R})+2\right\}x^{2}+\left\{-(1-p_{R})^{3}+3(1-p_{R})-1\right\}x+\frac{1}{2}(1-p_{R})^{3}\right].\label{g_{3}(x)-x_factorized}
\end{align}
Let us denote by $r(x)$ the quadratic polynomial within the square brackets of \eqref{g_{3}(x)-x_factorized}, so that $r(x)$ has the roots
\begin{align}
{}& \frac{(1-p_{R})^{3}-3(1-p_{R})+1 \pm \sqrt{\left\{-(1-p_{R})^{3}+3(1-p_{R})-1\right\}^{2}-\left\{(1-p_{R})^{3}-6(1-p_{R})+4\right\}(1-p_{R})^{3}}}{(1-p_{R})^{3}-6(1-p_{R})+4}\nonumber\\
={}& \frac{(1-p_{R})^{3}-3(1-p_{R})+1 \pm \sqrt{(2p_{R}-1)(p_{R}+1+\sqrt{3})\{p_{R}-(\sqrt{3}-1)\}}}{(1-p_{R})^{3}-6(1-p_{R})+4}.\label{quadratic_roots}
\end{align}
First, we note that $r(0) = 1/2(1-p_{R})^{3}$ and $r(1) = 1$, both of which are strictly positive when $p_{R} \in [0,1)$. If the leading coefficient, $1/2(1-p_{R})^{3}-3(1-p_{R})+2$, is strictly negative, which happens whenever $p_{R} \in [0,2-\sqrt{3})$, we see that $r(x) \rightarrow -\infty$ when $x \rightarrow \infty$ as well as when $x \rightarrow -\infty$. Therefore, the two roots of $r(x)$ will, in this case, be real and appear \emph{outside} of the interval $[0,1]$.

Let us now consider the case where the leading coefficient, $1/2(1-p_{R})^{3}-3(1-p_{R})+2$, is strictly positive, which is equivalent to focusing on $p_{R} \in (2-\sqrt{3},1]$. We need only worry about the case where the roots of $r(x)$ are real, i.e.\ where the discriminant $(2p_{R}-1)(p_{R}+1+\sqrt{3})\{p_{R}-(\sqrt{3}-1)\}\geqslant 0$ (since $g_{3}$ has a unique fixed point in $[0,1]$, which is $1$, when the roots of $r(x)$ are complex). This is true when $p_{R} \in (2-\sqrt{3},1/2] \cup [\sqrt{3}-1,1]$. As noted above, $r(0)$ and $r(1)$ are both strictly positive, so that for the roots of $r(x)$ to exist \emph{within} the interval $(0,1)$, we must have $\min\{r(x): x \in \mathbb{R}\}$, attained at 
\begin{equation}
\frac{(1-p_{R})^{3}-3(1-p_{R})+1}{2\left\{\frac{1}{2}(1-p_{R})^{3}-3(1-p_{R})+2\right\}} = \frac{(1-p_{R})^{3}-3(1-p_{R})+1}{(1-p_{R})^{3}-6(1-p_{R})+4},\nonumber
\end{equation}
\emph{inside} the interval $(0,1)$. Since the denominator has already been assumed to be positive, we now need to make sure that the numerator is positive as well, which happens whenever $p_{R} \in (0.6527,1]$ (note that the ratio above is already bounded above by $1$, which is why we need no further conditions on $p_{R}$ to ensure that $\min\{r(x): x \in \mathbb{R}\}$ is attained at a point strictly less than $1$). Note that the sets $(2-\sqrt{3},1/2] \cup [\sqrt{3}-1,1]$ and $(0.6527,1]$ overlap in $[\sqrt{3}-1,1]$, which tells us that $r(x)$ has two real roots, both within the interval $(0,1)$, if and only if $p_{R} \in [\sqrt{3}-1,1]$. 

Finally, note that when $p_{R}=2-\sqrt{3}$, the quadratic polynomial $r(x)$ reduces to a linear one, since the leading coefficient becomes $0$, but in this case, its only root, given by $1/2(1-p_{R})^{3}\{(1-p_{R})^{3}-3(1-p_{R})+1\}^{-1}$, will be negative since its denominator is negative.

When $p_{R} \in [0,\sqrt{3}-1)$, we know, by Remark~\ref{rem:unique_fixed_point_attractive}, that no matter which $\pi_{0} \in [0,1]$ we choose for the distribution in \eqref{initial}, the limit $\pi=\lim_{t \rightarrow \infty}\pi_{t}$, where $\pi_{t}$ is as defined in \eqref{time_t}, exists and equals the unique fixed point of $g_{3}$, which is $1$. Next, we consider $p_{R} \in (\sqrt{3}-1,1)$, so that $g_{3}$ has three distinct fixed points in $(0,1]$ (as proved above). As in the statement of Theorem~\ref{thm:unequal}, let us indicate the two fixed points of $g_{3}$ that are distinct from $1$ by $\alpha_{1}$ and $\alpha_{2}$, where $0 < \alpha_{1} < \alpha_{2} < 1$. Differentiating \eqref{g_{3}(x)_p_{B}=1}, we obtain
\begin{align}
g'_{3}(x) ={}& \frac{3}{2}(1-p_{R})\left(1+2p_{R}-p_{R}^{2}\right)(1-x)^{2}+6p_{R}x(1-x),\nonumber
\end{align}
implying that $g_{3}$ is a strictly increasing function on $[0,1]$, and since $p_{R} < 1$, we have $g_{3}(0) > 0$. We are now able to conclude, using Lemma~\ref{lem:increasing_function_limits}, that $\pi=\alpha_{1}$ whenever $\pi_{0} \in [0,\alpha_{2})$, whereas $\pi = 1$ whenever $\pi_{0} \in (\alpha_{2},1]$. Note that the case where $p_{R}=1$ is already taken care of in Theorem~\ref{thm:main_equal}, since $p_{B}=p_{R}=1$ in that case.

Finally, we consider $p_{R}=\sqrt{3}-1$. In this case, $g_{3}$ has precisely two distinct fixed points in $[0,1]$, namely $\alpha=2/3-1/\sqrt{3}$ and $1$, and the curve $y=g_{3}(x)$ merely \emph{touches} the line $y=x$ at $x=\alpha$ instead of intersecting it. Therefore, except for the points $x=\alpha$ and $x=1$, the curve $y=g_{3}(x)$ lies above the line $y=x$ throughout the entire interval $[0,1]$. Let us now choose $\pi_{0} \in [0,1) \setminus \{\alpha\}$. We have $\pi_{1} = g_{3}(\pi_{0}) > \pi_{0}$, and the strictly increasing nature of $g_{3}$ ensures that $\{\pi_{n}\}_{n \in \mathbb{N}_{0}}$ is a strictly increasing sequence. Consequently, the limit $\pi = \lim_{n \rightarrow \infty}\pi_{n}$ exists. When $\pi_{0} \in [0,\alpha)$, the strictly increasing nature of $g_{3}$ ensures that $0 \leqslant \pi_{0} < \pi_{n} < \alpha$ for each $n \in \mathbb{N}$, so that $0 < \pi \leqslant \alpha$, which immediately implies that $\pi=\alpha$. When $\pi_{0} \in (\alpha,1)$, we have $\alpha < \pi_{0} < \pi_{n} < 1$, so that $\alpha < \pi \leqslant 1$, which immediately implies $\pi=1$. This brings us to the end of the proof of Theorem~\ref{thm:unequal}.

\bibliography{Learning_model_bib}

\begin{thebibliography}{57}
\providecommand{\natexlab}[1]{#1}
\providecommand{\url}[1]{\texttt{#1}}
\expandafter\ifx\csname urlstyle\endcsname\relax
  \providecommand{\doi}[1]{doi: #1}\else
  \providecommand{\doi}{doi: \begingroup \urlstyle{rm}\Url}\fi

\bibitem[Bala and Goyal(1998)]{bala1998learning}
Venkatesh Bala and Sanjeev Goyal.
\newblock Learning from neighbours.
\newblock \emph{The review of economic studies}, 65\penalty0 (3):\penalty0
  595--621, 1998.

\bibitem[Bala and Goyal(2000)]{bala2000noncooperative}
Venkatesh Bala and Sanjeev Goyal.
\newblock A noncooperative model of network formation.
\newblock \emph{Econometrica}, 68\penalty0 (5):\penalty0 1181--1229, 2000.

\bibitem[Bandyopadhyay et~al.(2010)Bandyopadhyay, Roy, and
  Sarkar]{bandyopadhyay2010on}
Antar Bandyopadhyay, Rahul Roy, and Anish Sarkar.
\newblock On the one dimensional "learning from neighbours" model.
\newblock \emph{Electronic Journal of Probability}, 15, 2010.

\bibitem[Banerjee and Fudenberg(2004)]{banerjee2004word}
Abhijit Banerjee and Drew Fudenberg.
\newblock Word-of-mouth learning.
\newblock \emph{Games and economic behavior}, 46\penalty0 (1):\penalty0 1--22,
  2004.

\bibitem[Bernabeu et~al.(2011)Bernabeu, Calera-Rubio, I{\~n}esta, and
  Rizo]{bernabeu2011melodic}
Jos{\'e}~F Bernabeu, Jorge Calera-Rubio, Jos{\'e}~M I{\~n}esta, and David Rizo.
\newblock Melodic identification using probabilistic tree automata.
\newblock \emph{Journal of New Music Research}, 40\penalty0 (2):\penalty0
  93--103, 2011.

\bibitem[Bezuidenhout and Grimmett(1990)]{bezuidenhout1990critical}
Carol Bezuidenhout and Geoffrey Grimmett.
\newblock The critical contact process dies out.
\newblock \emph{The Annals of Probability}, 18\penalty0 (4):\penalty0
  1462--1482, 1990.

\bibitem[Bramson and Griffeath(1980)]{bramson1980williams}
Maury Bramson and David Griffeath.
\newblock On the williams-bjerknes tumour growth model: Ii.
\newblock In \emph{Mathematical Proceedings of the Cambridge Philosophical
  Society}, volume~88, pages 339--357. Cambridge University Press, 1980.

\bibitem[Bramson and Griffeath(1981)]{bramson1981williams}
Maury Bramson and David Griffeath.
\newblock On the williams-bjerknes tumour growth model: I.
\newblock 9:\penalty0 173--185, 1981.

\bibitem[Carayol et~al.(2014)Carayol, Haddad, and
  Serre]{carayol2014randomization}
Arnaud Carayol, Axel Haddad, and Olivier Serre.
\newblock Randomization in automata on infinite trees.
\newblock \emph{ACM Transactions on Computational Logic (TOCL)}, 15\penalty0
  (3):\penalty0 1--33, 2014.

\bibitem[Castellano et~al.(2009)Castellano, Mu{\~n}oz, and
  Pastor-Satorras]{castellano2009nonlinear}
Claudio Castellano, Miguel~A Mu{\~n}oz, and Romualdo Pastor-Satorras.
\newblock Nonlinear q-voter model.
\newblock \emph{Physical Review E}, 80\penalty0 (4):\penalty0 041129, 2009.

\bibitem[Chatterjee and Dutta(2016)]{chatterjee2016credibility}
Kalyan Chatterjee and Bhaskar Dutta.
\newblock Credibility and strategic learning in networks.
\newblock \emph{International Economic Review}, 57\penalty0 (3):\penalty0
  759--786, 2016.

\bibitem[Chatterjee and Xu(2004)]{chatterjee2004technology}
Kalyan Chatterjee and Susan~H Xu.
\newblock Technology diffusion by learning from neighbours.
\newblock \emph{Advances in Applied Probability}, 36\penalty0 (2):\penalty0
  355--376, 2004.

\bibitem[Clifford and Sudbury(1973)]{clifford1973model}
Peter Clifford and Aidan Sudbury.
\newblock A model for spatial conflict.
\newblock \emph{Biometrika}, 60\penalty0 (3):\penalty0 581--588, 1973.

\bibitem[Cohen et~al.(2009)Cohen, Kimelfeld, and Sagiv]{cohen2009running}
Sara Cohen, Benny Kimelfeld, and Yehoshua Sagiv.
\newblock Running tree automata on probabilistic xml.
\newblock In \emph{Proceedings of the twenty-eighth ACM SIGMOD-SIGACT-SIGART
  symposium on Principles of database systems}, pages 227--236, 2009.

\bibitem[Comon et~al.(2008)Comon, Dauchet, Gilleron, Jacquemard, Lugiez,
  L{\"o}ding, Tison, and Tommasi]{comon2008tree}
Hubert Comon, Max Dauchet, R{\'e}mi Gilleron, Florent Jacquemard, Denis Lugiez,
  Christof L{\"o}ding, Sophie Tison, and Marc Tommasi.
\newblock Tree automata techniques and applications, 2008.

\bibitem[Congleton(2004)]{congleton2004median}
Roger~D Congleton.
\newblock The median voter model.
\newblock In \emph{The encyclopedia of public choice}, pages 707--712.
  Springer, 2004.

\bibitem[De~Masi et~al.(1989)De~Masi, Presutti, and Scacciatelli]{de1989weakly}
Anna De~Masi, Errico Presutti, and E~Scacciatelli.
\newblock The weakly asymmetric simple exclusion process.
\newblock In \emph{Annales de l'IHP Probabilit{\'e}s et statistiques},
  volume~25, pages 1--38, 1989.

\bibitem[Durrett and Liu(1988)]{durrett1988contact}
Richard Durrett and Xiu-Fang Liu.
\newblock The contact process on a finite set.
\newblock \emph{The Annals of Probability}, pages 1158--1173, 1988.

\bibitem[Ellis(1970)]{ellis1970probabilistic}
Clarence~A Ellis.
\newblock Probabilistic tree automata.
\newblock In \emph{Proceedings of the second annual ACM symposium on Theory of
  computing}, pages 198--205, 1970.

\bibitem[Ellison and Fudenberg(1993)]{ellison1993rules}
Glenn Ellison and Drew Fudenberg.
\newblock Rules of thumb for social learning.
\newblock \emph{Journal of Political Economy}, 101\penalty0 (4):\penalty0
  612--643, 1993.

\bibitem[Ellison and Fudenberg(1995)]{ellison1995word}
Glenn Ellison and Drew Fudenberg.
\newblock Word-of-mouth communication and social learning.
\newblock \emph{The Quarterly Journal of Economics}, 110\penalty0 (1):\penalty0
  93--125, 1995.

\bibitem[Fudenberg and Levine(2009)]{fudenberg2009learning}
Drew Fudenberg and David~K Levine.
\newblock Learning and equilibrium.
\newblock \emph{Annu. Rev. Econ.}, 1\penalty0 (1):\penalty0 385--420, 2009.

\bibitem[Goyal(2011)]{goyal2011learning}
Sanjeev Goyal.
\newblock Learning in networks.
\newblock In \emph{Handbook of social economics}, volume~1, pages 679--727.
  Elsevier, 2011.

\bibitem[Goyal(2012)]{goyal2012connections}
Sanjeev Goyal.
\newblock \emph{Connections: an introduction to the economics of networks}.
\newblock Princeton University Press, 2012.

\bibitem[Granovsky and Madras(1995)]{granovsky1995noisy}
Boris~L Granovsky and Neal Madras.
\newblock The noisy voter model.
\newblock \emph{Stochastic Processes and their applications}, 55\penalty0
  (1):\penalty0 23--43, 1995.

\bibitem[Griffeath(1981)]{griffeath1981basic}
David Griffeath.
\newblock The basic contact processes.
\newblock \emph{Stochastic Processes and their Applications}, 11\penalty0
  (2):\penalty0 151--185, 1981.

\bibitem[Grimmett and Grimmett(1999)]{grimmett1999percolation}
Geoffrey Grimmett and Geoffrey Grimmett.
\newblock \emph{What is percolation?}
\newblock Springer, 1999.

\bibitem[Holley and Liggett(1975)]{holley1975ergodic}
Richard~A Holley and Thomas~M Liggett.
\newblock Ergodic theorems for weakly interacting infinite systems and the
  voter model.
\newblock \emph{The annals of probability}, pages 643--663, 1975.

\bibitem[Komarova(2006)]{komarova2006spatial}
Natalia~L Komarova.
\newblock Spatial stochastic models for cancer initiation and progression.
\newblock \emph{Bulletin of mathematical biology}, 68:\penalty0 1573--1599,
  2006.

\bibitem[Lamberson(2010)]{lamberson2010social}
PJ~Lamberson.
\newblock Social learning in social networks.
\newblock \emph{The BE Journal of Theoretical Economics}, 10\penalty0
  (1):\penalty0 0000102202193517041616, 2010.

\bibitem[Liggett(1976)]{liggett1976coupling}
Thomas~M Liggett.
\newblock Coupling the simple exclusion process.
\newblock \emph{The Annals of Probability}, pages 339--356, 1976.

\bibitem[Liggett(1997)]{liggett1997stochastic}
Thomas~M Liggett.
\newblock Stochastic models of interacting systems.
\newblock \emph{The Annals of Probability}, 25\penalty0 (1):\penalty0 1--29,
  1997.

\bibitem[Liggett(2013)]{liggett2013stochastic}
Thomas~M Liggett.
\newblock \emph{Stochastic interacting systems: contact, voter and exclusion
  processes}, volume 324.
\newblock springer science \& Business Media, 2013.

\bibitem[Liggett and Liggett(1985)]{liggett1985interacting}
Thomas~Milton Liggett and Thomas~M Liggett.
\newblock \emph{Interacting particle systems}, volume~2.
\newblock Springer, 1985.

\bibitem[Mallick(2011)]{mallick2011some}
Kirone Mallick.
\newblock Some exact results for the exclusion process.
\newblock \emph{Journal of Statistical Mechanics: Theory and Experiment},
  2011\penalty0 (01):\penalty0 P01024, 2011.

\bibitem[Masuda et~al.(2010)Masuda, Gibert, and
  Redner]{masuda2010heterogeneous}
Naoki Masuda, Nicolas Gibert, and Sidney Redner.
\newblock Heterogeneous voter models.
\newblock \emph{Physical Review E}, 82\penalty0 (1):\penalty0 010103, 2010.

\bibitem[Mele(2017)]{mele2017structural}
Angelo Mele.
\newblock A structural model of dense network formation.
\newblock \emph{Econometrica}, 85\penalty0 (3):\penalty0 825--850, 2017.

\bibitem[Mobilia(2003)]{mobilia2003does}
Mauro Mobilia.
\newblock Does a single zealot affect an infinite group of voters?
\newblock \emph{Physical review letters}, 91\penalty0 (2):\penalty0 028701,
  2003.

\bibitem[Mobilia and Georgiev(2005)]{mobilia2005voting}
Mauro Mobilia and Ivan~T Georgiev.
\newblock Voting and catalytic processes with inhomogeneities.
\newblock \emph{Physical Review E}, 71\penalty0 (4):\penalty0 046102, 2005.

\bibitem[Mobilia et~al.(2007)Mobilia, Petersen, and Redner]{mobilia2007role}
Mauro Mobilia, Anna Petersen, and Sidney Redner.
\newblock On the role of zealotry in the voter model.
\newblock \emph{Journal of Statistical Mechanics: Theory and Experiment},
  2007\penalty0 (08):\penalty0 P08029, 2007.

\bibitem[Mobius and Rosenblat(2014)]{mobius2014social}
Markus Mobius and Tanya Rosenblat.
\newblock Social learning in economics.
\newblock \emph{Annu. Rev. Econ.}, 6\penalty0 (1):\penalty0 827--847, 2014.

\bibitem[Mukhopadhyay et~al.(2020)Mukhopadhyay, Mazumdar, and
  Roy]{mukhopadhyay2020voter}
Arpan Mukhopadhyay, Ravi~R Mazumdar, and Rahul Roy.
\newblock Voter and majority dynamics with biased and stubborn agents.
\newblock \emph{Journal of Statistical Physics}, 181:\penalty0 1239--1265,
  2020.

\bibitem[Namatame(2010)]{namatame2010diffusion}
Akira Namatame.
\newblock Diffusion and emergence in social networks.
\newblock In \emph{Intelligent systems for automated learning and adaptation:
  Emerging trends and applications}, pages 231--247. IGI Global, 2010.

\bibitem[Namatame et~al.(2009)Namatame, Morita, and
  Matsuyama]{namatame2009agent}
Akira Namatame, Hiroshi Morita, and Kazuyuki Matsuyama.
\newblock Agent-based modeling for the study of diffusion dynamics.
\newblock In \emph{SpringSim}. Citeseer, 2009.

\bibitem[Quastel(1992)]{quastel1992diffusion}
Jeremy Quastel.
\newblock Diffusion of color in the simple exclusion process.
\newblock \emph{Communications on pure and applied mathematics}, 45\penalty0
  (6):\penalty0 623--679, 1992.

\bibitem[Rajewsky et~al.(1998)Rajewsky, Santen, Schadschneider, and
  Schreckenberg]{rajewsky1998asymmetric}
Nikolaus Rajewsky, Ludger Santen, Andreas Schadschneider, and Michael
  Schreckenberg.
\newblock The asymmetric exclusion process: Comparison of update procedures.
\newblock \emph{Journal of statistical physics}, 92:\penalty0 151--194, 1998.

\bibitem[Sandow(1994)]{sandow1994partially}
Sven Sandow.
\newblock Partially asymmetric exclusion process with open boundaries.
\newblock \emph{Physical review E}, 50\penalty0 (4):\penalty0 2660, 1994.

\bibitem[Sch{\"u}tz(1997)]{schutz1997exact}
Gunter~M Sch{\"u}tz.
\newblock Exact solution of the master equation for the asymmetric exclusion
  process.
\newblock \emph{Journal of statistical physics}, 88:\penalty0 427--445, 1997.

\bibitem[Sethi and Yildiz(2016)]{sethi2016communication}
Rajiv Sethi and Muhamet Yildiz.
\newblock Communication with unknown perspectives.
\newblock \emph{Econometrica}, 84\penalty0 (6):\penalty0 2029--2069, 2016.

\bibitem[Sood and Redner(2005)]{sood2005voter}
Vishal Sood and Sidney Redner.
\newblock Voter model on heterogeneous graphs.
\newblock \emph{Physical review letters}, 94\penalty0 (17):\penalty0 178701,
  2005.

\bibitem[Sood et~al.(2008)Sood, Antal, and Redner]{sood2008voter}
Vishal Sood, Tibor Antal, and Sidney Redner.
\newblock Voter models on heterogeneous networks.
\newblock \emph{Physical Review E}, 77\penalty0 (4):\penalty0 041121, 2008.

\bibitem[Tracy and Widom(2008)]{tracy2008integral}
Craig~A Tracy and Harold Widom.
\newblock Integral formulas for the asymmetric simple exclusion process.
\newblock \emph{Communications in Mathematical Physics}, 279\penalty0
  (3):\penalty0 815--844, 2008.

\bibitem[Tsakas(2017)]{tsakas2017diffusion}
Nikolas Tsakas.
\newblock Diffusion by imitation: The importance of targeting agents.
\newblock \emph{Journal of Economic Behavior \& Organization}, 139:\penalty0
  118--151, 2017.

\bibitem[Tsakas(2024)]{tsakas2024optimal}
Nikolas Tsakas.
\newblock Optimal influence under observational learning.
\newblock \emph{Mathematical Social Sciences}, 2024.

\bibitem[Watts(2001)]{watts2001dynamic}
Alison Watts.
\newblock A dynamic model of network formation.
\newblock \emph{Games and Economic Behavior}, 34\penalty0 (2):\penalty0
  331--341, 2001.

\bibitem[Williams and Bjerknes(1972)]{williams1972stochastic}
Trevor Williams and Rolf Bjerknes.
\newblock Stochastic model for abnormal clone spread through epithelial basal
  layer.
\newblock \emph{Nature}, 236\penalty0 (5340):\penalty0 19--21, 1972.

\bibitem[Young(2009)]{young2009innovation}
H~Peyton Young.
\newblock Innovation diffusion in heterogeneous populations: Contagion, social
  influence, and social learning.
\newblock \emph{American economic review}, 99\penalty0 (5):\penalty0
  1899--1924, 2009.

\end{thebibliography}
\end{document}